\documentclass[12pt]{amsart}
\usepackage[margin=1.in, centering]{geometry}

\usepackage{amssymb,latexsym,amsmath,extarrows, mathrsfs, amsthm}
\usepackage{graphicx}

\usepackage[colorlinks, linkcolor=blue]{hyperref}

\usepackage{color}
\usepackage{wasysym}

\usepackage{float}

\newtheorem{theorem}{Theorem}[section]
\newtheorem{lemma}[theorem]{Lemma}
\newtheorem{proposition}[theorem]{Proposition}

\theoremstyle{remark}

\usepackage{amsmath}

\newcommand{\ZZ}{\mathbb{Z}}

\allowdisplaybreaks
\title{Visible lattice points in higher dimensional random walks and biases among them}

\author{Kui Liu}
\address{School of Mathematics and Statistics, Qingdao University, 308 Ningxia Road, Shinan District, Qingdao, Shandong, China}
\email{liukui@qdu.edu.cn}

\author{Meijie Lu}
\address{School of Mathematics, Shandong University, Jinan 250100, Shandong, China}
\email{meijie.lu@hotmail.com}

\author{Xianchang Meng}
\address{School of Mathematics, Shandong University, Jinan 250100, Shandong, China}
\email{xianchang.meng@gmail.com}


\begin{document}

\maketitle

\begin{abstract}
For any integers $k\geq 2$, $q\geq 1$ and any finite set $\mathcal{A}=\{{\boldsymbol{\alpha}}_1,\cdots,{\boldsymbol{\alpha}}_q\}$, where ${ \boldsymbol{\alpha}_t}=(\alpha_{t,1},\cdots,\alpha_{t,k})~(1\leq t\leq q)$ with $0<\alpha_{t,1},\cdots,\alpha_{t,k}<1$ and $\alpha_{t,1}+\cdots+\alpha_{t,k}=1$, this paper concerns the visibility of lattice points in the type-$\mathcal{A}$ random walk on the lattice $\mathbb{Z}^k$.  We show that the proportion of visible lattice points on a random path of the walk is almost surely $1/\zeta(k)$, where $\zeta(s)$ is the Riemann zeta-function, and  we also consider consecutive visibility of lattice points in the type-$\mathcal{A}$ random walk and give the  proportion of the corresponding visible steps. Moreover, we find a new phenomenon that visible steps in both of the above cases are not evenly distributed. Our proof relies on tools from probability theory and analytic number theory.
\end{abstract}

\section{Introduction}

\subsection{Background}
Let $k\geq 2$ be any integer and ${\mathbb{Z}^{k}}$ be the $k$-dimensional integer lattice. A lattice point ${\bf n}=(n_1,\cdots,n_k)\in\mathbb{Z}^k$ is said to be \textit{visible} if there is no other lattice point on the straight line segment connecting ${\bf n}$ and the origin. Dirichlet \cite{D} (see also \cite{S}) showed that the density of visible lattice points in $\mathbb{Z}^2$ is $1/\zeta(2)$, where $\zeta(s)$ is the Riemann zeta function. Lehmer \cite{L} (see also \cite{C}) generalized Dirichlet's result to higher dimensional lattices. Visibility of lattice points along certain type of curves has also been considered, see for example \cite{BEH, GHKM, LM-2, LLM}. One may refer to \cite{CTZ, HO, R-thesis, R} for more related works.

The visibility of lattice points can also be considered from the view of random walks. Given integer $q\geq 1$ and any finite set $\mathcal{A}=\{{\boldsymbol{\alpha}}_1,\cdots,{\boldsymbol{\alpha}}_q\}$, where ${ \boldsymbol{\alpha}_t}=(\alpha_{t,1},\cdots,\alpha_{t,k})~(1\leq t\leq q)$ with $0<\alpha_{t,1},\cdots,\alpha_{t,k}<1$ and $\alpha_{t,1}+\cdots+\alpha_{t,k}=1$. On the lattice ${\mathbb{Z}^k}$, we define
a type-$\mathcal{A}$ random walk started from the origin ${\bf p}_0=(0,\cdots,0)$ by
\begin{align}\label{the definition of random walk}
{\bf p}_i={\bf p}_{i-1}+{\bf w}({\boldsymbol{\alpha}_i^{\prime}}),\ {\boldsymbol{\alpha}_i^{\prime}}\in \mathcal{A}
\end{align}
for $i=1,\ 2,\cdots$, where ${\bf p}_i=(p_{i,1},\cdots,p_{i,k})$ is the coordinate of the $i$-th step and
\begin{align*}
{\bf w}({\boldsymbol{\alpha}}):=\left\{
\begin{aligned}
&(1,0,\cdots,0),\ \ {\rm{with\ probability}}\ \alpha_1,\\
&(0,1,\cdots,0),\ \ {\rm{with\ probability}}\ \alpha_2,\\
&\quad\quad\quad\quad\cdots\\
&(0,0,\cdots,1),\ \ {\rm{with\ probability}}\ \alpha_k,\\
\end{aligned}
\right.
\end{align*}
is called a type-${\boldsymbol{\alpha}}$ step
 for some ${\boldsymbol{\alpha}}=(\alpha_1,\cdots,\alpha_k)$ randomly chosen from $\mathcal{A}$ at each step.
 When $\mathcal{A}=\{\boldsymbol{\alpha}\}$ has only one element, we simply call it a type-$\boldsymbol{\alpha}$ random walk.

 In 2015, Cilleruelo, Fern\'{a}ndez and Fern\'{a}ndez \cite{CFF} proved that the proportion of visible lattice points on a random path of a type-${\boldsymbol{\alpha}}$ random walk on $\mathbb{Z}^2$ is almost surely $1/\zeta(2)$. Recently, Liu and Meng \cite{LM-1} generalized their result to the cases of visible points along curves and multiple random walkers.

 In this paper, for any finite set $\mathcal{A}$ and any $k\geq 2$, we study the visibility of lattice points in the type-$\mathcal{A}$ random walk in $\ZZ^k$. In order to deal with all the higher dimensional cases $k\ge 2$, we propose a different approach than that in \cite{CFF} and give the proportion of visible steps in type-$\mathcal{A}$ random walks. Our results are more general that the random variables ${\bf w}({\boldsymbol{\alpha}})$ may have different distributions at every step, while previous results only concern ${\bf w}$ with the same distribution.

\subsection{Our results}
Associated to the type-$\mathcal{A}$ random walk defined in \eqref{the definition of random walk}, we consider a sequence of random variables $(X_i)_{i\geq1}$ with
$$
X_i:=\left\{
\begin{aligned}
&1,\ \ {\rm{if}}\ {\bf p}_i\ \rm{is\ visible},\\
&0,\ \ \rm{otherwise.}
\end{aligned}
\right.
$$
The random variable
$$
\overline{S}_{n,k}:=\frac{1}{n}(X_1+\cdots+X_n)
$$
indicates the proportion of visible steps in a type-$\mathcal{A}$ random walk in the first $n$ steps.
\begin{theorem}\label{thm:S_n=}
For any integer $k\geq 2$, we have
$$
\lim_{n\rightarrow\infty} \overline{S}_{n,k}=\frac{1}{\zeta(k)}
$$
almost surely, where $\zeta(s)$ is the Riemann zeta function.
\end{theorem}

When $k=2$ and $\mathcal{A}$ has only one element, Theorem \ref{thm:S_n=} gives Theorem A of \cite{CFF}.

For fixed integer $m\geq 2$, define $$
\overline{S}_{n,k}(a;m):=\frac{1}{n}\sum_{\substack{1\leq i\leq n\\i\equiv a(\bmod m)} }X_i
$$
for any integer $0\leq a\leq m-1$, then $\overline{S}_{n,k}(a;m)$ is the proportion of visible steps that are congruent to $a\bmod m$ in the first $n$ steps in a type-$\mathcal{A}$ random walk. Let
$$
\delta_{k}(a;m):=\lim\limits_{n\rightarrow\infty}\overline{S}_{n,k}(a;m),
$$
if the limit exists almost surely. One may expect that $\delta_{k}(a;m)$ is $1/m\zeta{(k)}$ for any $a$. However, we find that this is not true, namely, the visible steps are not evenly distributed. This is a surprisingly new phenomenon. We calculate some values of $\delta_k$ for certain types of $m$ and obtain the following results.
\begin{theorem}\label{thm:delta(a,m)=} For any integers $r\geq 1$ and $k\geq 2$, we have
\begin{align*}
\delta_{k}(a;2^r)=\begin{cases}
\dfrac{2^{k-r}}{2^k-1}\cdot\dfrac{1}{\zeta(k)},\quad &\text{if}~a~\text{is odd};\\
\dfrac{2^{k-1}-1}{2^{r-1}(2^k-1)}\cdot\dfrac{1}{\zeta(k)},\quad &\text{if}~a~\text{is even}
\end{cases}
\end{align*}
almost surely, and for any prime $p_1\geq 3$
\begin{align*}
\delta_{k}(a;p_1)=\begin{cases}\dfrac{p_1^{k-1}-1}{p_1^k-1}\cdot\dfrac{1}{\zeta(k)},\quad &\textit{if}~a=0;\\
\dfrac{p_1^{k-1}}{p_1^k-1}\cdot\dfrac{1}{\zeta(k)},\quad &\text{otherwise}
\end{cases}
\end{align*}
almost surely.
\end{theorem}

We are also interested in consecutive visible steps in the type-$\mathcal{A}$ random walk. Define the random variable
$$
\overline{R}_{n,k}:=\frac{1}{n}(X_1X_2+\cdots+X_nX_{n+1}).
$$
Then $\overline{R}_{n,k}$ is the proportion of two consecutive visible steps in a type-$\mathcal{A}$ random walk in the first $n+1$ steps.

\begin{theorem}\label{thm:R_n=}
For any integer $k\geq 2$, we have
$$
\lim_{n\rightarrow\infty}\overline{R}_{n,k}=\prod_p\bigg(1-\frac{2}{p^k}\bigg)
$$
almost surely, where $p$ runs over all primes.
\end{theorem}
When $k=2$ and  $\mathcal{A}$ has only one element, Theorem \ref{thm:R_n=} gives Theorem B of \cite{CFF}.

Similarly, for any fixed integer $m\geq 2$ and $0\leq a\leq m-1$, define $$
\overline{R}_{n,k}(a;m):=\frac{1}{n}\sum_{\substack{1\leq i\leq n\\i\equiv a(\bmod m)} }X_iX_{i+1},
$$
which indicates the proportion of two consecutive visible steps that are congruent to $a\bmod m$ and $a+1\bmod m$ in the first $n+1$ steps in a type-$\mathcal{A}$ random walk. Let
$$
\gamma_{k}(a;m):=\lim\limits_{n\rightarrow\infty}\overline{R}_{n,k}(a;m),
$$
if the limit exists almost surely.
We have the following results.
\begin{theorem}\label{thm:gamma_k{a;m}=}For any integers $r\geq 1$ and $k\geq 2$, we have, for all $0\leq a\leq 2^r-1$,
$$
\gamma_{k}(a;2^r)=\frac{1}{2^r}\prod_p\bigg(1-\frac{2}{p^k}\bigg)
$$
almost surely, and for any prime $p_1\geq 3$

\[
\gamma_{k}(a;p_1)=\begin{cases}
\dfrac{p_1^{k-1}-1}{p_1^k-2}{\displaystyle\prod_p}\bigg(1-\dfrac{2}{p^k}\bigg), & \text{if}~ a=0~\text{or}~p_1-1;\\

\dfrac{p_1^{k-1}}{p_1^k-2}{\displaystyle\prod_p}\bigg(1-\dfrac{2}{p^k}\bigg), &\text{otherwise}

\end{cases}
\]
almost surely, where $p$ runs over all primes.
\end{theorem}

Our results hold for any finite set $\mathcal{A}$, this means that at each step one may  use many different choices of walking strategies (finitely many). One may wonder whether we can relax the condition further to arbitrarily (infinitely) many choices of walking strategies at each step.

\textbf{Open Question:} Are these results  still true for an infinite set $\mathcal{A}$?

Now we explain a little bit about the difficulty of generalizing previous results to higher dimensions and our proof strategy. In the paper of Cilleruelo, Fern\'{a}ndez and Fern\'{a}ndez, they gave a key lemma (see \cite{CFF}, Lemma 2.1) needed to prove their theorems, which is a binomial theorem with a congruence condition. They used the local central limit theorem of the binomial distribution in probability theory to prove their lemma. However, this method is not ready to generalize to higher dimensions. In our paper, we use analytic method to handle higher dimensions and the  more complicated cases mentioned above (see Lemma \ref{lem:the sum in l dimensions}).

In the proof of Lemma \ref{lem:the sum in l dimensions}, we use the orthogonality of additive characters
$$
\frac{1}{d}\sum_{0\leq h\leq d-1}e\Big(\frac{hn}{d}\Big)=\left\{
\begin{aligned}
&1,\ \ \ \ {\rm{if}}\ d\mid n,\\
&0,\ \ \ \ \rm{otherwise}.
\end{aligned}
\right.
$$
to transform corresponding sums into an asymptotic formula (see formula \eqref{eq:L=}) with
$$
\sum_{1\leq h_{m_1},\cdots,h_{m_i}\leq d-1}\prod_{t=1}^q\Big|\alpha_{t,m_1}e\big(\frac{h_{m_1}}{d}\big)+\cdots+\alpha_{t,m_{i}}e\big(\frac{h_{m_i}}{d}\big)+\eta_{t,i}\Big|^{i_t}
$$
as an inner sum in the big-$O$ term, where $i_t$ is the number of type-${\boldsymbol{\alpha}}_t$ steps in the first $n$ steps in type-$\mathcal{A}$ random walk. The key point in our proof is that no matter how the random walker walks in the first $n$ steps, there always exists some $T\ (1\leq T\leq q)$ such that $i_{T}\geq \frac{n}{q}$ (This is also why we need $q$ to be finite). Hence the above formula becomes feasible to bound and then we obtain our result by applying certain number theoretic lemmas.

\bigskip

\textbf{Notations.} We use $\mathbb{Z}$ and $\mathbb{N}$ to denote the sets of integers and positive integers, respectively; use $\mathbb{P}(A)$ to denote the probability of an event $A$, use $\mathbb{E}(X)$ and $\mathbb{V}(X)$ to denote the expectation and variance of a random variable $X$; and use $[x]$ to denote the largest integer not exceeding the real number $x$. We also use the expressions $f=O(g)$ (or $f\ll g$) to mean $|f|\leq Cg$ for some constant $C>0$. When the constant $C$ depends on some parameters ${\bf \rho}$, we write $f=O_{\bf \rho}(g)$(or $f\ll_\rho g$).
\bigskip

\textbf{Acknowledgements.} The first listed author is partially supported by National Natural Science Foundation of China (NSFC, Grant No. 12071238) and Shandong Provincial
Natural Science Foundation (Grant No. ZR2019BA028).
The second and third listed authors are  supported by the National Natural Science Foundation of China (NSFC, Grant No. 12201346) and Shandong Provincial Foundation (Grant No. 2022HWYQ-046).    

\section{Preliminaries}
Throughout this paper, let $q\geq 1$ be an integer, we always denote
$s^{(a)}:=s_{1,a}+\cdots+s_{q,a}$
for integers $s_{1,a},\cdots,s_{q,a}$.

For integer $k\geq 2$ and nonnegative integers $n,\ u_1,\cdots,u_k$ with $u_1+\cdots+u_k=n$, we always write
$$
\binom{n}{u_1,\cdots,u_k}:=\frac{n!}{u_1!\cdots u_k!}.
$$
As a convention, the above formula is of value $1$ for $n=0$.

 Assume ${\boldsymbol{\alpha}}=(\alpha_1,\cdots,\alpha_k)$ with $0<\alpha_1,\cdots,\alpha_k<1$ and $\alpha_1+\cdots+\alpha_k=1$ and ${\bf u}=(u_1,\cdots,u_k)$, we denote
\begin{equation}\label{defn-Big-P}
P_{n,{\bf u},{\boldsymbol{\alpha}}}:=\binom{n}{u_1,\cdots,u_k}\alpha_1^{u_1}\cdots\alpha_k^{u_k}.
\end{equation}

We  use the following estimate concerning cosine function.

\begin{lemma}\label{lem:estimation of cos^lx}
For any integers $l\geq 1$ and $d>2$, we have
$$
\sum_{1\leq h<\frac{d}{2}}\cos^l\big(\frac{\pi h}{d}\big)=O\Big(\frac{d}{\sqrt{l}}\Big).
$$
\end{lemma}

\begin{proof}
Since $\cos^l(\pi t/d)$ is decreasing for $t\in[0,d/2]$,  we have
$$
\sum_{1\leq h<\frac{d}{2}}\cos^l\big(\frac{\pi h}{d}\big)\ll\int_{0}^{\frac{d}{2}}\cos^l\big(\frac{\pi t}{d}\big)dt.
$$
Using change of variable $x=\pi t/d$ and the formula
$$
\int_{0}^{\frac{\pi}{2}}\cos^l(x)dx=\frac{\sqrt{\pi}\Gamma(\frac{l+1}{2})}{2\Gamma(\frac{l+2}{2})},
$$
where $\Gamma(s)$ is Euler gamma function, we obtain
\begin{align}\label{eq: sum_hcos^l<<}
\sum_{1\leq h<\frac{d}{2}}\cos^l\big(\frac{\pi h}{d}\big)\ll d\int_{0}^{\frac{\pi}{2}}\cos^l(x)dx\ll\frac{d\Gamma(\frac{l+1}{2})}{\Gamma(\frac{l+2}{2})}.
\end{align}
Applying Stirling's formula
$$
\Gamma(s)=\sqrt{2\pi}s^{s-1/2}e^{-s}\big(1+O_{\varepsilon}(|s|^{-1})\big),
$$
for any $0<\varepsilon<\pi$ and $-\pi+\varepsilon<\arg s<\pi-\varepsilon$, we have
$$
\frac{\Gamma(\frac{l+1}{2})}{\Gamma(\frac{l+2}{2})}\ll\frac{(\frac{l+1}{2})^{\frac{l}{2}}}{(\frac{l+2}{2})^{\frac{l+1}{2}}}\ll\frac{1}{\sqrt{l}}.
$$
This together with \eqref{eq: sum_hcos^l<<} gives our desired result.
\end{proof}

By Lemma \ref{lem:estimation of cos^lx}, we prove the following  result that is useful for our proof.

\begin{lemma}\label{lem:the sum in two dimensions}
Suppose $0<\delta\leq \frac{1}{2}$ is fixed and $n\in\mathbb{N}$. For any integer $ d>1$, we have
$$
\frac{1}{d}\sum_{1\leq h\leq d-1}\big(1-\delta+\delta\cos(\frac{2\pi h}{d})\big)^{\frac{n}{2}}=O_{\delta}\big(\frac{\log n}{\sqrt{n}}\big)
$$
as $n\rightarrow\infty$.
\end{lemma}

\begin{proof}
In our proof, the implied constants in the big-$O$ and $\ll$  depend at most on $\delta$. It can be verified by direct calculation that the result is true for  $d=2$, so we only consider $d>2$. Denote
$$
H_n=H_n(d,\delta):=\frac{1}{d}\sum_{1\leq h\leq d-1}\big(1-\delta+\delta\cos(\frac{2\pi h}{d})\big)^{\frac{n}{2}}.
$$

If $n=2m$ is even, applying the binomial theorem and exchangeing the order of summations, we obtain
$$
H_n=\frac{1}{d}\sum_{0\leq l\leq m}\binom{m}{l}(1-\delta)^{m-l}\delta^l \sum_{1\leq h\leq d-1}\cos^l\big(\frac{2\pi h}{d}\big).
$$
By an elementary argument, we derive
$$
\sum_{1\leq h\leq d-1}\cos^l\big(\frac{2\pi h}{d}\big)=2\sum_{1\leq h<d/2}\cos^l\big(\frac{2\pi h}{d}\big)+\epsilon_{d}(-1)^l,
$$
where
$$
\epsilon_{d}=\begin{cases}
1, \quad{\rm if}\ d\ {\rm\ is\ even},\\
0, \quad{\rm if}\  d\ {\rm\ is\ odd}.
\end{cases}
$$
For $\sum\limits_{1\leq h<d/2}\cos^l\big(2\pi h/d\big)$, we use change of variable $h^{\prime}=2h$ and obtain
$$
\sum_{1\leq h<d/2}\cos^l\big(\frac{2\pi h}{d}\big)=\sum_{\substack{1\leq h^{\prime}<d\\h^{\prime}\ {\rm{even}}}}\cos^l\big(\frac{\pi h^{\prime}}{d}\big).
$$
Dividing the above sum into two parts, we derive
$$
\sum_{1\leq h<d/2}\cos^l\big(\frac{2\pi h}{d}\big)=\sum_{\substack{1\leq h^{\prime}<d/2\\h^{\prime}\ {\rm{even}}}}\cos^l\big(\frac{\pi h^{\prime}}{d}\big)+\sum_{\substack{d/2< h^{\prime}<d\\h^{\prime}\ {\rm{even}}}}\cos^l\big(\frac{\pi h^{\prime}}{d}\big).
$$
Letting $t=d-h^{\prime}$ for the second term on the right hand side, we obtain
$$
\sum_{1\leq h<d/2}\cos^l\big(\frac{2\pi h}{d}\big)=\sum_{\substack{1\leq h^{\prime}<d/2\\h^{\prime}\ {\rm{even}}}}\cos^l\big(\frac{\pi h^{\prime}}{d}\big)+(-1)^l\sum_{\substack{1\leq t<d/2\\t\ {\rm{even}}}}\cos^l\big(\frac{\pi t}{d}\big)
$$
for even $d$, and
$$
\sum_{1\leq h<d/2}\cos^l\big(\frac{2\pi h}{d}\big)=\sum_{\substack{1\leq h^{\prime}<d/2\\h^{\prime}\ {\rm{even}}}}\cos^l\big(\frac{\pi h^{\prime}}{d}\big)+(-1)^l\sum_{\substack{1\leq t<d/2\\t\ {\rm{odd}}}}\cos^l\big(\frac{\pi t}{d}\big)
$$
for odd $d$. Therefore
$$
\sum_{1\leq h<d/2}\cos^l\big(\frac{2\pi h}{d}\big)\ll\sum_{1\leq h<d/2}\cos^l\big(\frac{\pi h}{d}\big)
$$
for all $d>2$. Since the contribution of $\epsilon_{d}(-1)^l$ to $H_n$ is $\ll (1-2\delta)^m\ll \frac{1}{\sqrt{m}}$, then we have
$$
H_n\ll\frac{1}{d}\sum_{0\leq l\leq m}\binom{m}{l}(1-\delta)^{m-l}\delta^l I(l,d)+\frac{1}{\sqrt{m}},
$$
where
$$
I(l,d)=\sum_{1\leq h<d/2}\cos^l\big(\frac{\pi h}{d}\big).
$$
Dividing the sum over $l$ into two parts according to $l\leq m/\log^2m$ or not, we have
\begin{align}\label{eq:H_n<<H_n,1+H_n,2}
H_n\ll H_{n}^{\prime}+H_{n}^{\prime\prime}+\frac{1}{\sqrt{m}},
\end{align}
where
$$
H_{n}^{\prime}:=\frac{1}{d}\sum_{0\leq l\leq m/\log^2m}\binom{m}{l}(1-\delta)^{m-l}\delta^lI(l,d)
$$
and
$$
H_{n}^{\prime\prime}:=\frac{1}{d}\sum_{m/\log^2m< l\leq m}\binom{m}{l}(1-\delta)^{m-l}\delta^lI(l,d).
$$

For $H_{n}^{\prime}$, we estimate the sum over $h$ trivially and obtain
\begin{align*}
H_{n}^{\prime}\ll\sum_{0\leq l\leq m/\log^2m}\binom{m}{l}(1-\delta)^{m-l}\delta^l\ll(1-\delta)^m\sum_{0\leq l\leq m/\log^2m}\binom{m}{l},
\end{align*}
where we used the fact $\delta\leq \frac{1}{2}$. Observe that
$$
\binom{m}{l}=\frac{m!}{l!\ (m-l)!}\leq \frac{m^l}{l!},
$$
then we have
$$
H_{n}^{\prime}\ll(1-\delta)^mm^{m/\log^2m}\sum_{0\leq l\leq m/\log^2m}\frac{1}{l!}\ll(1-\delta)^mm^{m/\log^2m}.
$$
It follows that
\begin{align*}
(1-\delta)^mm^{m/\log^2m}=\big((1-\delta)e^{1/\log m}\big)^m\leq(1-\delta)^m\ll\frac{1}{\sqrt{m}}
\end{align*}
for $m$ being sufficiently large, which implies
\begin{align}\label{eq:H_n,1<<}
H_{n}^{\prime}\ll\frac{1}{\sqrt{n}}.
\end{align}

For $H_{n}^{\prime\prime}$, we apply Lemma \ref{lem:estimation of cos^lx} and obtain
$$
H_{n}^{\prime\prime}\ll\sum_{m/\log^2m< l\leq m}\binom{m}{l}\frac{(1-\delta)^{m-l}\delta^l}{\sqrt{l}}\ll\frac{\log m}{\sqrt m}\sum_{m/\log^2m< l\leq m}\binom{m}{l}(1-\delta)^{m-l}\delta^l\le \frac{\log m}{\sqrt m},
$$
which implies
\begin{align}\label{eq:H_n,2<<}
H_{n}^{\prime\prime}\ll\frac{\log n}{\sqrt{n}}.
\end{align}
Combining \eqref{eq:H_n,1<<}, \eqref{eq:H_n,2<<} and \eqref{eq:H_n<<H_n,1+H_n,2}, we derive
\begin{align}\label{eq:H_n<< for n is even}
H_n\ll\frac{\log n}{\sqrt{n}}
\end{align}
for even $n$.

If $n$ is odd, we observe that
$$
|1-\delta+\delta\cos(2\pi h/d)|\leq 1,
$$
then
$$
H_n\ll\frac{1}{d}\sum_{1\leq h\leq d-1}\big(1-\delta+\delta\cos\frac{2\pi h}{d}\big)^{\frac{n-1}{2}}=H_{n-1}.
$$
Since $n-1$ is even, we obtain
\begin{align}\label{eq:H_n<< for n is odd}
H_n\ll H_{n-1} \ll\frac{\log (n-1)}{\sqrt{n-1}}\ll\frac{\log n}{\sqrt{n}}
\end{align}
for odd $n$.

Now our desired result follows by combining \eqref{eq:H_n<< for n is even} and \eqref{eq:H_n<< for n is odd}.
\end{proof}
\begin{lemma}\label{lem:estimate of beta_{t_i}}
Let $n\geq 1,\ l\geq 2$ and $1\leq i\leq l-1$ be integers, and let vector $\overline{{\boldsymbol{\beta}}}=(\beta_{t_1},\cdots,\beta_{t_i},\eta_i)$ with $0<\beta_{t_1},\cdots,\beta_{t_i},\eta_i<1$ and $\eta_i=1-(\beta_{t_1}+\cdots+\beta_{t_i})$. Then for any integer $d>1$, we have
$$
\frac{1}{d^{l-1}}\sum_{1\leq h_{t_1},\cdots,h_{t_i}\leq d-1}\Big|\beta_{t_1}e\big(\frac{h_{t_1}}{d}\big)+\cdots+\beta_{t_{i}}e\big(\frac{h_{t_i}}{d}\big)+\eta_i\Big|^n=O_{\overline{{\boldsymbol{\beta}}}}\big(\frac{\log n}{\sqrt{n}}\big)
$$
as $n\rightarrow\infty$, where $e(x):=e^{2\pi ix}$ for any real number $x$.
\end{lemma}

\begin{proof}
In our proof, the implied constants in the big-$O$ and $\ll$  depend at most on $\overline{{\boldsymbol{\beta}}}$.
Denote
$$
J_n=J_n(d,l,\overline{{\boldsymbol{\beta}}}):=\frac{1}{d^{l-1}}\sum_{1\leq h_{t_1},\cdots,h_{t_i}\leq d-1}\Big|\beta_{t_1}e\big(\frac{h_{t_1}}{d}\big)+\cdots+\beta_{t_{i}}e\big(\frac{h_{t_i}}{d}\big)+\eta_i\Big|^n.
$$
Expanding the $n$-th power inside the above sum and using the formula $\cos(2\theta)=2\cos^2(\theta)-1$ for real number $\theta$, we obtain
$$
J_n\ll\frac{1}{d^{l-1}}\sum_{1\leq h_{t_1},\cdots,h_{t_i}\leq d-1}\Big(\sum_{1\leq a\leq i}\beta^2_{t_a}+\eta^2_i-2\sum_{1\leq a<b\leq i}\beta_{t_a}\beta_{t_b}-2\ \eta_i\sum_{1\leq b\leq i}\beta_{t_b}+J_n^{\prime}\Big)^{\frac{n}{2}},
$$
where
$$
J_n^{\prime}:=4\sum_{1\leq a<b\leq i}\beta_{t_a}\beta_{t_b}\cos^2\big(\frac{\pi(h_{t_a}-h_{t_b})}{d}\big)+4\eta_i\sum_{1\leq b\leq i}\beta_{t_b}\cos^2\big(\frac{\pi h_{t_b}}{d}\big).
$$
Note that
$$
J_n^{\prime}\ll 4\sum_{1\leq a<b\leq i}\beta_{t_a}\beta_{t_b}+4\eta_i\sum_{1\leq b\leq i-1}\beta_{t_b}+4\eta_i\beta_{t_i}\cos^2\big(\frac{\pi h_{t_i}}{d}\big),
$$
we have
$$
J_n\ll\frac{1}{d^{l-1}}\sum_{1\leq h_{t_1},\cdots,h_{t_i}\leq d-1}\Big(\sum_{1\leq a\leq i}\beta^2_{t_a}+\eta^2_i+2\sum_{1\leq a<b\leq i}\beta_{t_a}\beta_{t_b}+2\eta_i\sum_{1\leq b\leq i-1}\beta_{t_b}+2\eta_i\beta_{t_i}\cos\big(\frac{2\pi h_{t_i}}{d}\big)\Big)^{\frac{n}{2}},
$$
here we used $\cos(2\theta)=2\cos^2(\theta)-1$ again. It follows that
$$
J_n\ll\frac{1}{d}\sum_{1\leq h_{t_i}\leq d-1}
\Big(1-\delta+\delta\cos\big(\frac{2\pi h_{t_i}}{d}\big)\Big)^{\frac{n}{2}},
$$
where $\delta=2\eta_i\beta_{t_i}$.
Then our desired result follows from Lemma \ref{lem:the sum in two dimensions}.
\end{proof}

With the help of Lemma \ref{lem:estimate of beta_{t_i}}, we derive the following lemma, which is important to prove our theorem.

\begin{lemma}\label{lem:the sum in l dimensions}
Let $n$, $q\geq 1$, $k\geq 2$ and $1\leq t\leq q$ be integers and $g_1,\cdots,g_{k-1}\in\mathbb{Z}$. Suppose  ${\bf s}_t=(s_{t,1},\cdots,s_{t,k})\in\mathbb{Z}^k$, and ${\boldsymbol{\alpha}}_t=(\alpha_{t,1},\cdots,\alpha_{t,k})$ with $0<\alpha_{t,1},\cdots,\alpha_{t,k}<1$ and $\alpha_{t,1}+\cdots+\alpha_{t,k}=1$. For any integers $d\geq 1$ and $0\leq i_1,\cdots,i_q\leq n$ with $i_1+\cdots+i_q=n$, we have
$$
\sum_{\substack{0\leq s_{t,1},\cdots,s_{t,k}\leq i_t\\s_{t,1}+\cdots+s_{t,k}=i_t\\ \forall 1\leq t\leq q\\ s^{(a)}\equiv g_{a}(\bmod d),\\a=1,\cdots,k-1}}\prod_{t=1}^qP_{i_t,{\bf s}_t,{\boldsymbol{\alpha}}_t}=\frac{1}{d^{k-1}}+O\big(\frac{\log n}{\sqrt{n}}\big)
$$
 as $n\rightarrow\infty$, where $P_{i, \boldsymbol{s}, \boldsymbol{\alpha}}$ is defined in \eqref{defn-Big-P}, and $s^{(a)}:=\sum\limits_{1\leq v\leq q}s_{v,a}$,  and the implied constant in the big-$O$ depends on $k,\ q$, and $({\boldsymbol{\alpha}}_t)_{1\leq t\leq q}$.
\end{lemma}

\begin{proof}
For simplicity, denote
$$
L_{n,d,k}=L_{n,d,k}(g_a,i_t,\boldsymbol{\alpha}_t)_{1\leq a\leq k-1,1\leq t\leq q}:=\sum_{\substack{0\leq s_{t,1},\cdots,s_{t,k}\leq i_t\\s_{t,1}+\cdots+s_{t,k}=i_t\\ \forall 1\leq t\leq q\\ s^{(a)}\equiv g_a(\bmod d),\\a=1,\cdots,k-1}}\prod_{t=1}^qP_{i_t,{\bf s}_t,{\boldsymbol{\alpha}}_t}.
$$
Using the orthogonality of additive characters
$$
\frac{1}{d}\sum_{0\leq h\leq d-1}e\Big(\frac{hn}{d}\Big)=\left\{
\begin{aligned}
&1,\ \ \ \ {\rm{if}}\ d\mid n,\\
&0,\ \ \ \ \rm{otherwise},
\end{aligned}
\right.
$$
where $n\in \mathbb{Z}$ and $e(x):=e^{2\pi ix}$ for any $x\in \mathbb{R}$, we obtain
\begin{align*}
L_{n,d,k}=\frac{1}{d^{k-1}}\sum_{\substack{0\leq s_{t,1},\cdots,s_{t,k}\leq i_t\\s_{t,1}+\cdots+s_{t,k}=i_t\\ \forall 1\leq t\leq q}}\prod_{t=1}^qP_{i_t,{\bf s}_t,{\boldsymbol{\alpha}}_t}\prod_{a=1}^{k-1}\sum_{0\leq h_a\leq d-1}e\Big(\frac{h_a(s^{(a)}-g_a)}{d}\Big).
\end{align*}
Changing the order of the summations and applying the binomial theorem, we derive
$$
L_{n,d,k}=\frac{1}{d^{k-1}}\sum_{0\leq h_1,\cdots, h_{k-1}\leq d-1}\prod_{a=1}^{k-1}e(-\frac{h_ag_a}{d})\prod_{t=1}^q\Big(\alpha_{t,1}e\big(\frac{h_1}{d}\big)+\cdots+\alpha_{t,k-1}e\big(\frac{h_{k-1}}{d}\big)+\alpha_{t,k}\Big)^{i_t}.
$$
Taking out the term $h_1=h_2=\cdots =h_{k-1}=0$, we get
\begin{align}\label{eq:L=}
L_{n,d,k}=\frac{1}{d^{k-1}}+O\big(L_{n,d,k}^{\prime}\big),
\end{align}
where
$$
L_{n,d,k}^{\prime}:=\frac{1}{d^{k-1}}\sum_{1\leq i\leq k-1}\sum_{\substack{1\leq m_1,\cdots,m_i\leq k-1\\m_1<\cdots<m_i}}\sum_{1\leq h_{m_1},\cdots,h_{m_i}\leq d-1}\prod_{t=1}^q\Big|\alpha_{t,m_1}e\big(\frac{h_{m_1}}{d}\big)+\cdots+\alpha_{t,m_{i}}e\big(\frac{h_{m_i}}{d}\big)+\eta_{t,i}\Big|^{i_t}
$$
with $\eta_{t,i}=1-(\alpha_{t,m_1}+\cdots+\alpha_{t,m_{i}})$.

Since $i_1+\cdots+i_q=n$,  there exists $i_T\geq \frac{n}{q}$ for some $1\leq T\leq q$. Thus
$$
L_{n,d,k}^{\prime}\ll\frac{1}{d^{k-1}}\sum_{1\leq i\leq k-1}\sum_{\substack{1\leq m_1,\cdots,m_i\leq k-1\\m_1<\cdots<m_i}}\sum_{1\leq h_{m_1},\cdots,h_{m_i}\leq d-1}\Big|\alpha_{T,m_1}e\big(\frac{h_{m_1}}{d}\big)+\cdots+\alpha_{T,m_{i}}e\big(\frac{h_{m_i}}{d}\big)+\eta_{T,i}\Big|^{i_T},
$$
here we used the bound
$$
\Big|\alpha_{t,m_1}e\big(\frac{h_{m_1}}{d}\big)+\cdots+\alpha_{t,m_{i}}e\big(\frac{h_{m_i}}{d}\big)+\eta_{t,i}\Big|\leq \alpha_{t,m_1}+\cdots + \alpha_{t,m_{i}}+\eta_{t,i}=1
$$
for any $1\leq t\leq q$. Hence we have
$$
L_{n,d,k}^{\prime}\ll\frac{\log i_T}{\sqrt{i_T}}\sum_{1\leq i\leq k-1}\sum_{\substack{1\leq m_1,\cdots,m_i\leq k-1\\m_1<\cdots<m_i}}1\ll \frac{\log n}{\sqrt{n}},
$$
thanks to Lemma \ref{lem:estimate of beta_{t_i}}, which together with \eqref{eq:L=} implies our required result .
\end{proof}

The next lemma is the criteria of visibility of lattice points in $\mathbb{Z}^k$.

\begin{lemma}{\rm{(\cite{R-thesis},\ p.\ 2)}}\label{lem:the critia of visibility}
For any integer $k\geq 2$ and ${\bf n}=(n_1,\cdots,n_k)\in\mathbb{Z}^k$, the lattice point ${\bf n}$ is visible  if and only if $\gcd(n_1,\cdots,n_k)=1$.
\end{lemma}

We need the following key lemma from probability, which is the second moment method.

\begin{lemma}{\rm{(\cite{CFF},\ Lemma\ 2.5)}}\label{lem:lim S_n=u}
 Let $(X_i)_{i\geq1}$ be a sequence of uniformly bounded random variables such that
$$
\lim_{n\rightarrow\infty}\mathbb{E}(\overline{S}_n)=\mu
$$
where
$$
\overline{S}_n=\frac{1}{n}\sum_{1\leq i\leq n}X_i.
$$
If there exists a constant $\delta>0$ such that the variance $\mathbb{V}(\overline{S}_n)=O(n^{-\delta})$ for $n\geq1$, then we have
$$
\lim_{n\rightarrow\infty}\overline{S}_n=\mu
$$
almost surely.
\end{lemma}

For the divisor function $\tau(i)=\sum\limits_{d\mid i}1$, we know that $\tau(i)\ll_\varepsilon i^{\varepsilon}$ for any $\varepsilon>0$ (see \cite{A}, p.296). Moreover, we need the following estimates in our proofs, we state them here without a proof, where (i) can be refered to Lemma 2.4 of \cite{CFF}.
\begin{lemma}\label{lem:the estimate about tau(n)}
{\rm{(i)}} As $n\rightarrow\infty$,$$
\sum_{1\leq i\leq n}\frac{\tau(i)}{\sqrt{i}}=O(\sqrt{n}\log n),\ \ \sum_{1\leq i<j\leq n}\frac{\tau(j)}{\sqrt{j-i}}=O(n^{3/2}\log n).
$$
{\rm{(ii)}} For any integer $n\geq 2$ and real number $\theta>1$,
  $$
  \sum_{1\leq i\leq n}\frac{\tau(i)}{i}=O(\log^2 n),\ \ \sum_{i>n}\frac{\tau(i)}{i^\theta}=O_\theta(n^{1-\theta}\log n).
  $$
\end{lemma}

\begin{lemma}\label{lem:the eatimate about [n/d]}
Assume $l\geq 2$ is an integer and $n\in\mathbb{N}$. We have
$$
\sum_{1\leq d\leq n}\frac{\mu(d)}{d^{l-1}}\left[\frac{n}{d}\right]=\frac{n}{\zeta(l)}+E(n),
$$
as $n\rightarrow\infty$, where ${\mu}(d)$ is the M\"{o}bius function and
$$
E(n)=\left\{
\begin{aligned}
&O_{l}(\log n),\ \ \ &{\rm{if}}\ l=2,\\
&O_{l}(1),\ \ \ &{\rm{if}}\ l>2.
\end{aligned}
\right.
$$
\end{lemma}

\begin{proof}
We write
\begin{align}\label{eq:sum_{d<=n}u(d)/d^{l-1}[n/d]=}
\sum_{1\leq d\leq n}\frac{\mu(d)}{d^{l-1}}\left[\frac{n}{d}\right]=n\sum_{1\leq d\leq n}\frac{\mu(d)}{d^{l}}+O\Big(\sum_{1\leq d\leq n}\frac{1}{d^{l-1}}\Big),
\end{align}
where we used the bound $|\mu(d)|\leq 1$ for $d\in\mathbb{N}$. Applying the formula
$$
\sum_{1\leq n\leq x}\frac{1}{n}=\log x+\gamma+O(x^{-1}),\ \ \ \text{as}~x\rightarrow\infty,
$$
where $\gamma$ is the Euler's constant, we obtain
\begin{align}\label{eq:estimate of_{d<=n}1/d^{l-1}}
\sum_{1\leq d\leq n}\frac{1}{d^{l-1}}=\left\{
\begin{aligned}
&O(\log n),\ \ &{\rm{if}}\ l=2,\\
&O(1),\ \ &{\rm{if}} \ l>2.
\end{aligned}
\right.
\end{align}
Since the series $\sum_{d=1}^\infty\mu(d)/d^l$ is convergent, we have
$$
\sum_{1\leq d\leq n}\frac{\mu(d)}{d^{l}}=\sum_{d=1}^\infty\frac{\mu(d)}{d^{l}}+O\big(\sum_{d>n}\frac{1}{d^{l}}\big).
$$
Using the well-known formula
\begin{align}\label{eq:sum_{d>D}1/d^{1-theta}<<}
\sum_{d>D}\frac{1}{d^\theta}=O_{\theta}(D^{1-\theta})
\end{align}
for real number $\theta>1$ and $D\geq 1$, we have
\begin{align}\label{eq:sum_{d<=n}u(d)/d^{l}=}
\sum_{1\leq d\leq n}\frac{\mu(d)}{d^{l}}=\sum_{d=1}^\infty\frac{\mu(d)}{d^{l}}+O_l(n^{1-l})=\frac{1}{\zeta(l)}+O_l(n^{1-l}).
\end{align}
Collecting \eqref{eq:sum_{d<=n}u(d)/d^{l-1}[n/d]=}, \eqref{eq:estimate of_{d<=n}1/d^{l-1}} and \eqref{eq:sum_{d<=n}u(d)/d^{l}=} gives our result.
\end{proof}

\begin{lemma}\label{lem:sum_{d_1d_2<n}=}
Suppose $n\geq1$ and $l\geq2$ are integers, then for integers $d_1$, $d_2\geq1$, we have
$$
\sum_{1\leq i\leq n}\sum_{\substack{1\leq d_1d_2\leq i\\d_1\mid i,d_2\mid i+1\\ \gcd(d_1,d_2)=1}}\frac{\mu(d_1)\mu(d_2)}{(d_1d_2)^{l-1}}=n\prod_p\big(1-\frac{2}{p^l}\big)+O_{l,\varepsilon}(n^\varepsilon), ~\text{as}~ n\rightarrow\infty.
$$

\end{lemma}

\begin{proof}
Denote
$$
V_{n,l}:=\sum_{1\leq i\leq n}\sum_{\substack{1\leq d_1d_2\leq i\\d_1\mid i,d_2\mid i+1\\ \gcd(d_1,d_2)=1}}\frac{\mu(d_1)\mu(d_2)}{(d_1d_2)^{l-1}}.
$$
Changing the order of the summations, we derive
$$
V_{n,l}=\sum_{\substack{1\leq d_1d_2\leq n\\ \gcd(d_1,d_2)=1}}\frac{\mu(d_1)\mu(d_2)}{(d_1d_2)^{l-1}}\sum_{\substack{d_1d_2\leq i\leq n\\ d_1\mid i,d_2\mid i+1}}1.
$$
Using the Chinese Reminder Theorem, we obtain
$$
V_{n,l}=\sum_{\substack{1\leq d_1d_2\leq n\\ \gcd(d_1,d_2)=1}}\frac{\mu(d_1)\mu(d_2)}{(d_1d_2)^{l-1}}\Big(\frac{n}{d_1d_2}+O(1)\Big).
$$
Letting $d_1d_2=h$ and using Lemma \ref{lem:the estimate about tau(n)}, we obtain
$$
V_{n,l}=n\sum_{1\leq h\leq n}\frac{\mu(h)\tau(h)}{h^{l}}+O(n^\varepsilon),
$$
here we  used the formula that $\log x=O_\varepsilon(x^\varepsilon)$ as $x\rightarrow\infty$, $\forall \varepsilon>0$.

Since the series $\sum_{h=1}^\infty\mu(h)\tau (h)/h^l$ is convergent,  we have
$$
V_{n,l}=n\sum_{h=1}^\infty\frac{\mu(h)\tau(h)}{h^{l}}+ O\Big(n\sum_{h>n}\frac{\tau(h)}{h^{l}}\Big)+O(n^\varepsilon).
$$
With the help of Lemma \ref{lem:the estimate about tau(n)}, we obtain
$$
V_{n,l}=n\sum_{h=1}^\infty\frac{\mu(h)\tau(h)}{h^{l}}+O(n^\varepsilon).
$$
Then our result follows by the fact that $\mu(h)$ and $\tau(h)$ are multiplicative.
\end{proof}

\section{Proof of Theorem \ref{thm:S_n=}}

By Lemma \ref{lem:lim S_n=u}, we only consider the expectation and variance of $\overline{S}_{n,k}$. In this section, the implied constants in the big-$O$ and $\ll$ depend at most on $\mathcal{A}$, $k,\ q$ and any given $\varepsilon>0$.
\begin{proposition}\label{prop:E(s(n))=}
Let $q\geq 1$ and $k\geq 2$ be integers. Then for any $\varepsilon>0$, we have
$$
\mathbb{E}(\overline{S}_{n,k})=\frac{1}{\zeta(k)}+O_{\mathcal{A},k,q,\varepsilon}(n^{-1/2+\varepsilon})
$$
as $n\rightarrow\infty$.
\end{proposition}

\begin{proof}
By the definition of ${X}_i$, we have
$$
\mathbb{E}({X}_i)=\mathbb{P}({\bf p}_i\ \rm{is\ visible}).
$$
In the first $i$ steps in a type-$\mathcal{A}$ random walk, suppose type-${\boldsymbol{\alpha}}_t$ step is totally $i_t$ steps, then $0\leq i_1,\cdots,i_q\leq i$ and $i_1+\cdots+i_q=i$. In these $i_t$ steps, suppose there are $s_{t,1}$ steps in the direction of $(1,0,\cdots,0)$, $\cdots$, $s_{t,k}$ steps in the direction of $(0,0,\cdots,1)$. Thus the coordinate of ${\bf p}_i$ is of the form
$$
{\bf p}_i=(s^{(1)},\cdots,s^{(k)}),
$$
where $s^{(a)}:=s_{1,a}+\cdots+s_{q,a}$ ($1\leq a\leq k$) for some $0\leq s_{t,1},\cdots,s_{t,k}\leq i_t$ and $s_{t,1}+\cdots+s_{t,k}=i_t$, $1\leq t\leq q$. The probability that ${\bf p}_i$ is of such form is
$$
\prod_{t=1}^q P_{i_t,{\bf s}_t,{\boldsymbol{\alpha}}_t}=\prod_{t=1}^q\binom{i_t}{s_{t,1},\cdots,s_{t,k}}\alpha_{t,1}^{s_{t,1}}\cdots \alpha_{t,k}^{s_{t,k}}.
$$
It follows that
\begin{align}\label{eq:E(X_i)=(1)}
\mathbb{E}({X}_i)=\sum_{\substack{0\leq s_{t,1},\cdots,s_{t,k}\leq i_t\\s_{t,1}+\cdots+s_{t,k}=i_t\\ \forall 1\leq t\leq q\\ (s^{(1)},\cdots,s^{(k)})\ \rm{is\ visible}}}\prod_{t=1}^q P_{i_t,{\bf s}_t,{\boldsymbol{\alpha}}_t},
\end{align}
where ${\bf s}_t=(s_{t,1},\cdots,s_{t,k})$. By Lemma \ref{lem:the critia of visibility}, we have
$$
\mathbb{E}({X}_i)=\sum_{\substack{0\leq s_{t,1},\cdots,s_{t,k}\leq i_t\\s_{t,1}+\cdots+s_{t,k}=i_t\\ \forall 1\leq t\leq q\\ \gcd(s^{(1)},\cdots,s^{(k)})=1}}\prod_{t=1}^q P_{i_t,{\bf s}_t,{\boldsymbol{\alpha}}_t}.
$$
Using the M\"{o}bius inversion formula
\begin{align}\label{eq:mobius formula}
\sum_{d\mid n}\mu(d)=\left\{
\begin{aligned}
&1,\ \ {\rm if}\ n=1,\\
&0,\ \ {\rm otherwise},
\end{aligned}
\right.
\end{align}
we obtain
$$
\mathbb{E}({X}_i)=\sum_{\substack{0\leq s_{t,1},\cdots,s_{t,k}\leq i_t\\s_{t,1}+\cdots+s_{t,k}=i_t\\ \forall 1\leq t\leq q}}\prod_{t=1}^q P_{i_t,{\bf s}_t,{\boldsymbol{\alpha}}_t}\sum_{d\mid\gcd(s^{(1)},\cdots,s^{(k)})}\mu(d).
$$
By the definition of $\gcd(\ast)$, we have
$$
\mathbb{E}({X}_i)=\sum_{\substack{0\leq s_{t,1},\cdots,s_{t,k}\leq i_t\\s_{t,1}+\cdots+s_{t,k}=i_t\\ \forall 1\leq t\leq q}}\prod_{t=1}^q P_{i_t,{\bf s}_t,{\boldsymbol{\alpha}}_t}\sum_{\substack{d\mid s^{(a)},\\ \forall1\leq a\leq k}}\mu(d).
$$
Changing the order of the summations gives us
$$
\mathbb{E}({X}_i)=\sum_{1\leq d\leq i}\mu(d)\sum_{\substack{0\leq s_{t,1},\cdots,s_{t,k}\leq i_t\\s_{t,1}+\cdots+s_{t,k}=i_t\\ \forall 1\leq t\leq q\\d\mid s^{(a)},\forall1\leq a\leq k}}\prod_{t=1}^q P_{i_t,{\bf s}_t,{\boldsymbol{\alpha}}_t},
$$
and furthermore, we have
\begin{align}\label{eq:the equation of EX_i about gcd}
\mathbb{E}({X}_i)=\sum_{d\mid i}\mu(d)\sum_{\substack{0\leq s_{t,1},\cdots,s_{t,k}\leq i_t\\s_{t,1}+\cdots+s_{t,k}=i_t\\ \forall 1\leq t\leq q\\d\mid s^{(a)},\forall1\leq a\leq k-1}}\prod_{t=1}^q P_{i_t,{\bf s}_t,{\boldsymbol{\alpha}}_t}.
\end{align}
By Lemma \ref{lem:the sum in l dimensions}, we deduce
\begin{align}\label{eq:E(X_i)=(2)}
\nonumber\mathbb{E}({X}_i)&=\sum_{d\mid i}\mu(d)\Big(\frac{1}{d^{k-1}}+O(i^{-1/2+\varepsilon})\Big)\\
&=\sum_{d\mid i}\frac{\mu(d)}{d^{k-1}}+O\left(i^{-1/2+\varepsilon}\tau(i)\right)
\end{align}
for any $\varepsilon>0$, where $\tau(i)$ is the divisor function.

By the definition of $\overline{S}_{n,k}$ and the additivity of the expectation, we have
$$
\mathbb{E}(\overline{S}_{n,k})=\frac{1}{n}\sum_{1\leq i\leq n}\mathbb{E}({X}_i)=\frac{1}{n}\sum_{1\leq i\leq n}\sum_{d\mid i}\frac{\mu(d)}{d^{k-1}}+O\Big(\frac{1}{n}\sum_{1\leq i\leq n}i^{-1/2+\varepsilon}\tau(i)\Big),
$$
which gives
$$
\mathbb{E}(\overline{S}_{n,k})=\frac{1}{n}\sum_{1\leq d\leq n}\frac{\mu(d)}{d^{k-1}}\left[\frac{n}{d}\right]+O\bigg(n^{-1+\varepsilon}\sum_{1\leq i\leq n}\frac{\tau(i)}{\sqrt{i}}\bigg).
$$
Then our desired result follows from Lemma \ref{lem:the eatimate about [n/d]} and Lemma \ref{lem:the estimate about tau(n)}.
 \end{proof}

 Now we estimate the variance of $\overline{S}_{n,k}$.
\begin{proposition}\label{prop:V(s(n))=}
Let $q\geq1$ and $k\geq 2$ be integers. Then for any $\varepsilon>0$, we have
$$
\mathbb{V}(\overline{S}_{n,k})=O_{\mathcal{A},k,q,\varepsilon}(n^{-1/2+\varepsilon})
$$
as $n\rightarrow\infty$.
\end{proposition}

\begin{proof}
To deal with $\mathbb{V}(\overline{S}_{n,k})$, we write
\begin{align}\label{eq:V(s(n))=E(s(n))^2-(E(s(n)))^2}
\mathbb{V}(\overline{S}_{n,k})=\mathbb{E}\big((\overline{S}_{n,k})^2\big)-\big(\mathbb{E}(\overline{S}_{n,k})\big)^2.
\end{align}
By Proposition \ref{prop:E(s(n))=}, we have
\begin{align}\label{eq:(E(s(n)))^2=}
\big(\mathbb{E}(\overline{S}_{n,k})\big)^2=\Big(\frac{1}{\zeta(k)}+O(n^{-1/2+\varepsilon})\Big)^2=\Big(\frac{1}{\zeta(k)}\Big)^2+O(n^{-1/2+\varepsilon}).
\end{align}
To compute $\mathbb{E}\big((\overline{S}_{n,k})^2\big)$, we expand the square and obtain
\begin{align}\label{eq:E((s(n))^2)=sum_E(X_iX_j)+sum_E(X_i^2)}
\mathbb{E}\big((\overline{S}_{n,k})^2\big)=\frac{2}{n^2}\sum_{1\leq i<j\leq n}\mathbb{E}({X}_i{X}_j)+\frac{1}{n^2}\sum_{1\leq i\leq n}\mathbb{E}({{X}_i}^2).
\end{align}
Observe that
\begin{align}\label{eq:sum EX_i^2=}
\sum_{1\leq i\leq n}\mathbb{E}({{X}_i}^2)=\sum_{1\leq i\leq n}\mathbb{E}({X}_i)=O(n).
\end{align}
For $\sum\limits_{1\leq i<j\leq n}\mathbb{E}({X}_i{X}_j)$, by the definition of ${X}_i$, we have
$$
\mathbb{E}({X}_i{X}_j)=\mathbb{P}({\bf p}_i\ {\rm{and}}\ {\bf p}_j\ {\rm{are\ both\ visible}}).
$$
By a similar argument as that before \eqref{eq:E(X_i)=(1)}, for $0\leq i_1,\cdots,i_q\leq i$ with $i_1+\cdots+i_q=i$ and $0\leq j_1,\cdots,j_q\leq j-i$ with $j_1+\cdots+j_q=j-i$, we see that the probability of
\begin{align*}
  {\bf p}_i&=(s^{(1)},\cdots,s^{(k)}),  \\
  {\bf p}_j&=(s^{(1)}+r^{(1)},\cdots,s^{(k)}+r^{(k)})
\end{align*}
is (here it means step $j$ depends on step $i$)
$$
\prod_{t=1}^q\binom{i_t}{s_{t,1},\cdots,s_{t,k}}\alpha_{t,1}^{s_{t,1}}\cdots \alpha_{t,k}^{s_{t,k}}
\binom{j_t}{r_{t,1},\cdots,r_{t,k}}\alpha_{t,1}^{r_{t,1}}\cdots \alpha_{t,k}^{r_{t,k}},
$$
where $0\leq s_{t,1},\cdots,s_{t,k}\leq i_t$  with $s_{t,1}+\cdots+s_{t,k}=i_t$, and $0\leq r_{t,1},\cdots,r_{t,k}\leq j_t$ with $r_{t,1}+\cdots+r_{t,k}=j_t,\ \forall1\leq t\leq q$.

It follows that
$$
\mathbb{E}({X}_i{X}_j)=\sum_{\substack{0\leq s_{t,1},\cdots,s_{t,k}\leq i_t\\s_{t,1}+\cdots+s_{t,k}=i_t\\ \forall 1\leq t\leq q\\ (s^{(1)},\cdots,s^{(k)})\ \rm{is\ visible}}}\prod_{t=1}^q P_{i_t,{\bf s}_t,{\boldsymbol{\alpha}}_t}\sum_{\substack{0\leq r_{t,1},\cdots,r_{t,k}\leq j_t\\r_{t,1}+\cdots+r_{t,k}=j_t\\ \forall 1\leq t\leq q\\ (s^{(1)}+r^{(1)},\cdots,s^{(k)}+r^{(k)})\ \rm{is\ visible}}}\prod_{t=1}^q P_{j_t,{\bf r}_t,{\boldsymbol{\alpha}}_t},
$$
where ${\bf r}_t=(r_{t,1},\cdots,r_{t,k})$. Further, by Lemma \ref{lem:the critia of visibility}, we have
\begin{align}\label{eq:E(X_iX_j)=}
\mathbb{E}({X}_i{X}_j)=\sum_{\substack{0\leq s_{t,1},\cdots,s_{t,k}\leq i_t\\s_{t,1}+\cdots+s_{t,k}=i_t\\ \forall 1\leq t\leq q\\ \gcd(s^{(1)},\cdots,s^{(k)})=1}}\prod_{t=1}^q P_{i_t,{\bf s}_t,{\boldsymbol{\alpha}}_t}\sum_{\substack{0\leq r_{t,1},\cdots,r_{t,k}\leq j_t\\r_{t,1}+\cdots+r_{t,k}=j_t\\ \forall 1\leq t\leq q\\ \gcd(s^{(1)}+r^{(1)},\cdots,s^{(k)}+r^{(k)})=1}}\prod_{t=1}^q P_{j_t,{\bf r}_t,{\boldsymbol{\alpha}}_t}.
\end{align}
By similar argument as that yields \eqref{eq:the equation of EX_i about gcd}, we obtain
$$
\mathbb{E}({X}_i{X}_j)=\bigg(\sum_{d\mid i}\mu(d)\sum_{\substack{0\leq s_{t,1},\cdots,s_{t,k}\leq i_t\\s_{t,1}+\cdots+s_{t,k}=i_t\\ \forall 1\leq t\leq q\\d\mid s^{(a)},\forall1\leq a\leq k-1}}\prod_{t=1}^q P_{i_t,{\bf s}_t,{\boldsymbol{\alpha}}_t}\bigg)
\bigg(\sum_{d\mid j}\mu(d)\sum_{\substack{0\leq r_{t,1},\cdots,r_{t,k}\leq j_t\\r_{t,1}+\cdots+r_{t,k}=j_t\\ \forall 1\leq t\leq q\\d\mid s^{(a)}+r^{(a)},\forall1\leq a\leq k-1}}\prod_{t=1}^q P_{j_t,{\bf r}_t,{\boldsymbol{\alpha}}_t}\bigg).
$$
Applying Lemma \ref{lem:the sum in l dimensions}, we have
$$
\mathbb{E}({X}_i{X}_j)=\Big(\sum_{d\mid i}\frac{\mu(d)}{d^{k-1}}+O\left(i^{-1/2+\varepsilon}\tau(i)\right)\Big)\Big(\sum_{d\mid j}\frac{\mu(d)}{d^{k-1}}+O\left((j-i)^{-1/2+\varepsilon}\tau(j)\right)\Big).
$$
With the help of the bound
\begin{align}\label{eq:sum_{d|i}u(d)/d^{k-1}<<}
\sum_{d\mid i}\frac{\mu(d)}{d^{k-1}}\ll\sum_{1\leq d\leq i}\frac{1}{d}\ll\log i
\end{align}
and $\tau(i)/\sqrt{i}\leq C$ for some $C>0$ and all $i$ (see \cite{HW}, Theorem 315), we obtain
$$
\mathbb{E}({X}_i{X}_j)=\sum_{d\mid i}\frac{\mu(d)}{d^{k-1}}\sum_{d\mid j}\frac{\mu(d)}{d^{k-1}}+O\left(i^{-1/2+\varepsilon}j^\varepsilon\tau(i)\right)+O\left(i^\varepsilon(j-i)^{-1/2+\varepsilon}\tau(j)\right).
$$
Hence
$$
\sum_{1\leq i<j\leq n}\mathbb{E}({X}_i{X}_j)=\sum_{1\leq i<j\leq n}\sum_{d\mid i}\frac{\mu(d)}{d^{k-1}}\sum_{d\mid j}\frac{\mu(d)}{d^{k-1}}+O(n^{3/2+\varepsilon}),
$$
here we  used Lemma \ref{lem:the estimate about tau(n)} to bound the error term.

For the first term on the right hand side of the above equation, we may add diagonal terms up to an error term $\ll n^{1+\varepsilon}$ by \eqref{eq:sum_{d|i}u(d)/d^{k-1}<<} and obtain
\begin{equation}\label{eq:E(X_iX_j)=(1)}
\sum_{1\leq i<j\leq n}\mathbb{E}({X}_i{X}_j)=\frac{1}{2}\bigg(\sum_{1\leq i\leq n}\sum_{d\mid i}\frac{\mu(d)}{d^{k-1}}\bigg)^2+O(n^{3/2+\varepsilon}).
\end{equation}
By Lemma \ref{lem:the eatimate about [n/d]}, we derive
\begin{align}\label{eq:sum EX_iX_j=}
\begin{aligned}
\sum_{1\leq i<j\leq n}\mathbb{E}({X}_i{X}_j)&=\frac{1}{2}\bigg(\sum_{1\leq d\leq n}\frac{\mu(d)}{d^{k-1}}\left[\frac{n}{d}\right]\bigg)^2+O(n^{3/2+\varepsilon})\\
&=\frac{1}{2}\Big(\frac{n}{\zeta(k)}\Big)^2+O(n^{3/2+\varepsilon}).
\end{aligned}
\end{align}
Plugging \eqref{eq:sum EX_i^2=} and \eqref{eq:sum EX_iX_j=} into \eqref{eq:E((s(n))^2)=sum_E(X_iX_j)+sum_E(X_i^2)}, we have
\begin{align}\label{eq:E((s(n))^2)=}
\mathbb{E}\big((\overline{S}_{n,k})^2\big)=\Big(\frac{1}{\zeta(k)}\Big)^2+O(n^{-1/2+\varepsilon}).
\end{align}
Inserting \eqref{eq:(E(s(n)))^2=} and \eqref{eq:E((s(n))^2)=} into \eqref{eq:V(s(n))=E(s(n))^2-(E(s(n)))^2} yields our desired result.

Now Theorem \ref{thm:S_n=} follows from Propositions \ref{prop:E(s(n))=} and \ref{prop:V(s(n))=}, and Lemma \ref{lem:lim S_n=u}.
\end{proof}

\section{Proof of Theorem \ref{thm:delta(a,m)=}}

 Similar to the proof of Theorem \ref{thm:S_n=}, we need to compute the expectation and variance of $\overline{S}_{n,k}(a;m)$.
In this section, the implied constants in the big-$O$ or $\ll$ depend at most on $\mathcal{A},\ k,\ q$ and any given $\varepsilon>0$.
\bigskip

We first deal with the expectation of $\overline{S}_{n,k}(a;2^r)$.
\begin{proposition}\label{prop:E(S(a,2^l))=}
Let $q,\ r\geq 1$ and $k\geq 2$ be integers. Then for any $\varepsilon>0$, we have
$$
\mathbb{E}\big(\overline{S}_{n,k}(a;2^r)\big)=\frac{2^{k-r}}{2^{k}-1}\cdot\frac{1}{\zeta(k)}+O_{\mathcal{A},k,q,\varepsilon}(n^{-1/2+\varepsilon})\quad\text{if }\ a\ \text{is odd}
$$
and
$$
\mathbb{E}\big(\overline{S}_{n,k}(a;2^r)\big)=\frac{2^{k-1}-1}{2^{r-1}(2^{k}-1)}\cdot\frac{1}{\zeta(k)}+O_{\mathcal{A},k,q,\varepsilon}(n^{-1/2+\varepsilon})\quad\text{if }\ a\ \text{is even}
$$
as $n\rightarrow\infty$.
\end{proposition}

\begin{proof}
It follows by \eqref{eq:E(X_i)=(2)} that
$$
\mathbb{E}({X}_i)=\sum_{d\mid i}\frac{\mu(d)}{d^{k-1}}+O\left(i^{-1/2+\varepsilon}\tau(i)\right).
$$
Hence by Lemma \ref{lem:the estimate about tau(n)}, we have
\begin{equation}\label{eq:E(S(a,2^l))=}
\mathbb{E}\big(\overline{S}_{n,k}(a;2^r)\big)=T_{n,k}(a;2^r)+O(n^{-1/2+\varepsilon}),
\end{equation}
where
$$
T_{n,k}(a;2^r):=\frac{1}{n}\sum_{\substack{1\leq i\leq n\\i\equiv a(\bmod 2^r)}}\sum_{d\mid i}\frac{\mu(d)}{d^{k-1}}.
$$
For $T_{n,k}(a;2^r)$, we change the order of the summations,
$$
T_{n,k}(a;2^r)=\frac{1}{n}\sum_{1\leq d\leq n}\frac{\mu(d)}{d^{k-1}}\sum_{\substack{1\leq i\leq n\\i\equiv a(\bmod 2^r)\\i\equiv0(\bmod d)}}1.
$$
Note that the system of congruence equations
\begin{align}\label{eq:system of congruence equation}
\begin{cases}
i\equiv a(\bmod\ 2^r)\\
i\equiv 0(\bmod\ d)
\end{cases}
\end{align}
has solutions if and only if $\gcd(2^r,d)\mid a$, in which case, it has solutions
 $i\equiv c_1\big(\bmod\ \text{lcm}(2^r,d)\big)$ for some $c_1\in\mathbb{Z}$. 

If $a=0$, then by \eqref{eq:sum_{d|i}u(d)/d^{k-1}<<}, we derive
$$
\begin{aligned}
T_{n,k}(0;2^r)&=\frac{1}{n}\sum_{1\leq d\leq n}\frac{\mu(d)}{d^{k-1}}\Big(\frac{\gcd(2^r,d)n}{2^rd}+O(1)\Big)\\
&=\frac{1}{2^r}\sum_{1\leq d\leq n}\frac{\gcd(2^r,d)\mu(d)}{d^{k}}+O(n^{-1+\varepsilon}).
\end{aligned}
$$
Since the series $\sum_{d=1}^\infty \gcd(2^r,d)\mu(d)/d^k$ is convergent, by \eqref{eq:sum_{d>D}1/d^{1-theta}<<}, we have
$$
\begin{aligned}
T_{n,k}(0;2^r)&=\frac{1}{2^r}\sum_{d=1}^\infty\frac{\gcd(2^r,d)\mu(d)}{d^{k}}+O\big(\sum_{d>n}\frac{1}{d^k}\big)+O(n^{-1+\varepsilon})\\
&=\frac{1}{2^r}\sum_{d=1}^\infty\frac{\gcd(2^r,d)\mu(d)}{d^{k}}+O(n^{-1+\varepsilon}).
\end{aligned}
$$
Since $\gcd(2^r,d)$ and $\mu(d)$ are multiplicative in terms of $d$, we derive
\begin{equation}\label{eq:T(a,2^l) for a=0(2)}
T_{n,k}(0;2^r)=\frac{2^{k-1}-1}{2^{r-1}(2^{k}-1)}\cdot\frac{1}{\zeta(k)}+O(n^{-1+\varepsilon}).
\end{equation}

If $a$ is odd, then \eqref{eq:system of congruence equation} has solutions if and only if  $\gcd(2^r,d)=1$,  we have
$$
T_{n,k}(a;2^r)=\frac{1}{n}\sum_{\substack{1\leq d\leq n\\ \gcd(2^r,d)=1}}\frac{\mu(d)}{d^{k-1}}\Big(\frac{n}{2^rd}+O(1)\Big)
=\frac{1}{2^r}\sum_{\substack{1\leq d\leq n\\ \gcd(2^r,d)=1}}\frac{\mu(d)}{d^{k}}+O(n^{-1+\varepsilon}).
$$
If we define the function
$$
f(d):=\left\{
\begin{aligned}
&\frac{\mu(d)}{d^{k}},\ \ {\rm{if}}\ \gcd(2^r,d)=1,\\
&0,\ \ \ \ \ \ \ \rm{otherwise},
\end{aligned}
\right.
$$
then we have
$$
T_{n,k}(a;2^r)=\frac{1}{2^r}\sum_{1\leq d\leq n}f(d)+O(n^{-1+\varepsilon})=\frac{1}{2^r}\sum_{d=1}^\infty f(d)+O\big(\sum_{d>n}\frac{1}{d^k}\big)+O(n^{-1+\varepsilon}).
$$
Hence by \eqref{eq:sum_{d>D}1/d^{1-theta}<<} and that  $f(d)$ is multiplicative, we obtain
\begin{equation}\label{eq:T(a,2^l) for a odd}
T_{n,k}(a;2^r)=\frac{2^{k-r}}{2^{k}-1}\cdot\frac{1}{\zeta(k)}+O(n^{-1+\varepsilon}).
\end{equation}

If $a$ is even ($a\neq 0$), we suppose $\gcd(2^r,a)=2^b,\ 1\leq b\leq r-1$, then \eqref{eq:system of congruence equation} has solutions if and only if $2^{b+1}\nmid d$, thus
\begin{equation}\label{eq:T(a,2^l) for a even}
T_{n,k}(a;2^r)=\frac{1}{n}\sum_{\substack{1\leq d\leq n\\ 2^{b+1}\nmid d}}\frac{\mu(d)}{d^{k-1}}\Big(\frac{\gcd(2^r,d)n}{2^rd}+O(1)\Big)=T_{n,k}(0;2^r),
\end{equation}
here we used the fact that $\mu(d)=0$ for $2^{b+1}\mid d$.

Then our results follow by \eqref{eq:E(S(a,2^l))=}, \eqref{eq:T(a,2^l) for a=0(2)}, \eqref{eq:T(a,2^l) for a odd} and \eqref{eq:T(a,2^l) for a even}.
\end{proof}

Now we compute the expectation of $\overline{S}_{n,k}(a;p_1)$.
\begin{proposition}\label{prop:E(S(a,p_1))=}
Let $q\geq 1$ and $k\geq 2$ be integers. Then for any prime $p_1\geq 3$ and any $\varepsilon>0$, we have
$$
\mathbb{E}\big(\overline{S}_{n,k}(0;p_1)\big)=\frac{p_1^{k-1}-1}{p_1^{k}-1}\cdot\frac{1}{\zeta(k)}+O_{\mathcal{A},k,q,\varepsilon}(n^{-1/2+\varepsilon})
$$
and
$$
\mathbb{E}\big(\overline{S}_{n,k}(a;p_1)\big)=\frac{p_1^{k-1}}{p_1^{k}-1}\cdot\frac{1}{\zeta(k)}+O_{\mathcal{A},k,q,\varepsilon}(n^{-1/2+\varepsilon})\quad \text{if}~a\neq 0
$$
as $n\rightarrow\infty$.
\end{proposition}

\begin{proof}
Similar as \eqref{eq:E(S(a,2^l))=}, we have
\begin{equation}\label{eq:E(S(a,p_1))=}
\mathbb{E}\big(\overline{S}_{n,k}(a;p_1)\big)=T_{n,k}(a;p_1)+O(n^{-1/2+\varepsilon}),
\end{equation}
where
$$
T_{n,k}(a;p_1):=\frac{1}{n}\sum_{\substack{1\leq i\leq n\\i\equiv a(\bmod p_1)}}\sum_{d\mid i}\frac{\mu(d)}{d^{k-1}}.
$$
If $a=0$, then by a similar argument as that yields \eqref{eq:T(a,2^l) for a=0(2)}, we obtain
\begin{equation}\label{eq:T(a,p_1) for a=0}
T_{n,k}(0;p_1)=\frac{p_1^{k-1}-1}{p_1^{k}-1}\cdot\frac{1}{\zeta(k)}+O(n^{-1+\varepsilon}).
\end{equation}
If $a\neq 0$, then similar as \eqref{eq:T(a,2^l) for a odd}, there holds
\begin{equation}\label{eq:T(a,p_1) for a not 0}
T_{n,k}(a;p_1)=\frac{p_1^{k-1}}{p_1^{k}-1}\cdot\frac{1}{\zeta(k)}+O(n^{-1+\varepsilon}).
\end{equation}
Inserting \eqref{eq:T(a,p_1) for a=0} and \eqref{eq:T(a,p_1) for a not 0} into \eqref{eq:E(S(a,p_1))=} yields our required results.
\end{proof}

Now we compute the variance of $\overline{S}_{n,k}(a;m)$ for $m=2^r$ or $p_1$.
\begin{proposition}\label{prop:V(S(a,m))=}
Let $q,\ r\geq 1$ and $k\geq 2$ be integers. Then for any $\varepsilon>0$, we have
$$
\mathbb{V}\big(\overline{S}_{n,k}(a;m)\big)=O_{\mathcal{A},k,q,\varepsilon}(n^{-1/2+\varepsilon})
$$
as $n\rightarrow\infty$.
\end{proposition}

\begin{proof}
 Similar as in the proof of Proposition \ref{prop:V(s(n))=}, we only need to handle $\big(\mathbb{E}\big(\overline{S}_{n,k}(a;m)\big)\big)^2$ and $\mathbb{E}\big(\big(\overline{S}_{n,k}(a;m)\big)^2\big)$.

 By Proposition \ref{prop:E(S(a,2^l))=} and Proposition \ref{prop:E(S(a,p_1))=}, we write
 $$
 \big(\mathbb{E}\big(\overline{S}_{n,k}(a;m)\big)\big)^2=\big(T_{n,k}(a;m)\big)^2+O(n^{-1/2+\varepsilon}).
 $$
 For $\mathbb{E}\big(\big(\overline{S}_{n,k}(a;m)\big)^2\big)$, notice that
$$
\mathbb{E}\big(\big(\overline{S}_{n,k}(a;m)\big)^2\big)=\frac{2}{n^2}\sum_{\substack{1\leq i<j\leq n\\i,j\equiv a(\bmod m)}}\mathbb{E}({X}_i{X}_j)+\frac{1}{n^2}\sum_{\substack{1\leq i\leq n\\i\equiv a(\bmod m)}}\mathbb{E}({{X}_i}^2).
$$
By Propositions \ref{prop:E(S(a,2^l))=} and \ref{prop:E(S(a,p_1))=}, the second term on the right hand side of the above formula is $\ll 1/n$, and by a similar argument as that yields \eqref{eq:E(X_iX_j)=(1)}, we have
$$
\begin{aligned}
\frac{2}{n^2}\sum_{\substack{1\leq i<j\leq n\\i,j\equiv a(\bmod m)}}\mathbb{E}({X}_i{X}_j)&=\bigg(\frac{1}{n}\sum_{\substack{1\leq i\leq n\\i\equiv a(\bmod m)}}\sum_{d\mid i}\frac{\mu(d)}{d^{k-1}}\bigg)^2+O(n^{-1/2+\varepsilon})\\
&=\big(T_{n,k}(a,m)\big)^2+O(n^{-1/2+\varepsilon}),
\end{aligned}
$$
which gives
$$
\mathbb{E}\big(\big(\overline{S}_{n,k}(a;m)\big)^2\big)=\big(T_{n,k}(a,m)\big)^2+O(n^{-1/2+\varepsilon}).
$$
This completes our proof.

 Now Theorem \ref{thm:delta(a,m)=} follows from Propositions \ref{prop:E(S(a,2^l))=}, \ref{prop:E(S(a,p_1))=} and \ref{prop:V(S(a,m))=}, and Lemma \ref{lem:lim S_n=u}.
\bigskip
\end{proof}

\section{Proof of Theorem \ref{thm:R_n=}}
Similarly, we only consider the expectation and variance of $\overline{R}_{n,k}$.
If not specified, the implied constants in the big-$O$ or $\ll$ depend at most on $\mathcal{A},\ k,\ q$ and any given $\varepsilon>0$.

\begin{proposition}\label{prop:E(R_n)=}
Let $q\geq 1$ and $k\geq 2$ be integers. Then for any $\varepsilon>0$, we have
$$
\mathbb{E}(\overline{R}_{n,k})=\prod_p\Big(1-\frac{2}{p^k}\Big)+O_{\mathcal{A},k,q,\varepsilon}(n^{-1/2+\varepsilon})
$$
as $n\rightarrow\infty$, where $p$ runs over all primes.
\end{proposition}

\begin{proof}
By the definition of $\overline{R}_{n,k}$, we have
\begin{align}\label{eq:E(R_n)=}
\mathbb{E}(\overline{R}_{n,k})=\frac{1}{n}\sum_{1\leq i\leq n}\mathbb{E}({X}_i{X}_{i+1}).
\end{align}
Then by \eqref{eq:E(X_iX_j)=}, we obtain
$$
\mathbb{E}({X}_i{X}_{i+1})=\sum_{\substack{0\leq s_{t,1},\cdots,s_{t,k}\leq i_t\\s_{t,1}+\cdots+s_{t,k}=i_t\\ \forall 1\leq t\leq q\\ \gcd(s^{(1)},\cdots,s^{(k)})=1}}\prod_{t=1}^q P_{i_t,{\bf s}_t,{\boldsymbol{\alpha}}_t}\sum_{\substack{0\leq r_{t,1},\cdots,r_{t,k}\leq j_t\\r_{t,1}+\cdots+r_{t,k}=j_t\\ \forall 1\leq t\leq q\\ \gcd(s^{(1)}+r^{(1)},\cdots,s^{(k)}+r^{(k)})=1}}\prod_{t=1}^q P_{j_t,{\bf r}_t,{\boldsymbol{\alpha}}_t},
$$
where $0\leq j_1,\cdots,j_q\leq 1$ and $j_1+\cdots+j_q=1$. Without loss of generality, we suppose $j_1=1$, then we obtain
\begin{align}\label{eq:E(X_iX_i+1)=(1)}
\mathbb{E}({X}_i{X}_{i+1})=A_1+\cdots+A_k,
\end{align}
where
$$
A_a=A_a(\mathcal{A},k,i):=\alpha_{1,a}\sum_{\substack{0\leq s_{t,1},\cdots,s_{t,k}\leq i_t\\s_{t,1}+\cdots+s_{t,k}=i_t\\ \forall 1\leq t\leq q\\ \gcd(s^{(1)},\cdots,s^{(k)})=1\\ \gcd(s^{(1)},\cdots,s^{(a)}+1,\cdots,s^{(k)})=1}}\prod_{t=1}^q P_{i_t,{\bf s}_t,{\boldsymbol{\alpha}}_t}
$$
for  $1\leq a\leq k$.

We only consider $A_1(\mathcal{A},k,i)$,  and other cases are similar.
Using \eqref{eq:mobius formula}, we have
$$
A_1=\alpha_{1,1}\sum_{\substack{0\leq s_{t,1},\cdots,s_{t,k}\leq i_t\\s_{t,1}+\cdots+s_{t,k}=i_t\\ \forall 1\leq t\leq q}}\prod_{t=1}^q P_{i_t,{\bf s}_t,{\boldsymbol{\alpha}}_t}\sum_{d_1\mid\gcd(s^{(1)},\cdots,s^{(k)})}\mu(d_1)\sum_{d_2\mid\gcd(s^{(1)}+1,\cdots,s^{(k)})}\mu(d_2).
$$
By the definition of $\gcd(\ast)$, we have
$$
A_1=\alpha_{1,1}\sum_{\substack{0\leq s_{t,1},\cdots,s_{t,k}\leq i_t\\s_{t,1}+\cdots+s_{t,k}=i_t\\ \forall 1\leq t\leq q}}\prod_{t=1}^q P_{i_t,{\bf s}_t,{\boldsymbol{\alpha}}_t}\sum_{\substack{d_1\mid s^{(a)},\\ \forall 1\leq a\leq k}}\mu(d_1)\sum_{\substack{d_2\mid s^{(1)}+1,\\d_2\mid s^{(a)}, \forall 2\leq a\leq k}}\mu(d_2).
$$
We change the order of the summations, then
$$
A_1=\alpha_{1,1}\sum_{1\leq d_1,d_2\leq i}\mu(d_1)\mu(d_2)\sum_{\substack{0\leq s_{t,1},\cdots,s_{t,k}\leq i_t\\s_{t,1}+\cdots+s_{t,k}=i_t\\ \forall 1\leq t\leq q\\ d_1\mid s^{(1)},d_2\mid s^{(1)}+1\\ d_1\mid s^{(a)},d_2\mid s^{(a)},\forall 2\leq a\leq k}}\prod_{t=1}^q P_{i_t,{\bf s}_t,{\boldsymbol{\alpha}}_t}.
$$
Further, we have
$$
A_1=\alpha_{1,1}\sum_{\substack{1\leq d_1,d_2\leq i\\d_1\mid i,d_2\mid i+1}}\mu(d_1)\mu(d_2)\sum_{\substack{0\leq s_{t,1},\cdots,s_{t,k}\leq i_t\\s_{t,1}+\cdots+s_{t,k}=i_t\\ \forall 1\leq t\leq q\\ d_1\mid s^{(1)},d_2\mid s^{(1)}+1\\ d_1\mid s^{(a)},d_2\mid s^{(a)},\forall 2\leq a\leq k-1}}\prod_{t=1}^q P_{i_t,{\bf s}_t,{\boldsymbol{\alpha}}_t}.
$$

Observe that the $k-1$ systems of congruence equations
$$
\begin{array}{cc}
\begin{cases}
s^{(1)}\equiv 0(\bmod\ d_1)\\
s^{(1)}\equiv -1(\bmod\ d_2)
\end{cases}, &
\begin{cases}
s^{(a)}\equiv 0(\bmod\ d_1)\\
s^{(a)}\equiv 0(\bmod\ d_2)
\end{cases},
\end{array}
a=2,\cdots,k-1
$$
have solutions if and only if $\gcd(d_1,d_2)\mid 1$, this implies $\gcd(d_1,d_2)=1$. Applying the Chinese Reminder Theorem, we obtain
$$
A_1=\alpha_{1,1}\sum_{\substack{1\leq d_1,d_2\leq i\\d_1\mid i,d_2\mid i+1\\ \gcd(d_1,d_2)=1}}\mu(d_1)\mu(d_2)\sum_{\substack{0\leq s_{t,1},\cdots,s_{t,k}\leq i_t\\s_{t,1}+\cdots+s_{t,k}=i_t\\ \forall 1\leq t\leq q\\s^{(a)}\equiv g_a(\bmod d_1d_2)\\ \forall 1\leq a\leq k-1}}\prod_{t=1}^q P_{i_t,{\bf s}_t,{\boldsymbol{\alpha}}_t},
$$
where $g_a\in\mathbb{Z}$ for any $1\leq a\leq k-1$.
Using Lemma \ref{lem:the sum in l dimensions}, we have
$$
\begin{aligned}
A_1&=\alpha_{1,1}\sum_{\substack{1\leq d_1,d_2\leq i\\d_1\mid i,d_2\mid i+1\\ \gcd(d_1,d_2)=1}}\mu(d_1)\mu(d_2)\Big(\frac{1}{(d_1d_2)^{k-1}}+O(i^{-1/2+\varepsilon})\Big)\\
&=\alpha_{1,1}\sum_{\substack{1\leq d_1,d_2\leq i\\d_1\mid i,d_2\mid i+1\\ \gcd(d_1,d_2)=1}}\frac{\mu(d_1)\mu(d_2)}{(d_1d_2)^{k-1}}+O\big(i^{-1/2+\varepsilon}\tau(i)\tau(i+1)\big).
\end{aligned}
$$
Dividing the above sum into two parts according to $d_1d_2\leq i$ or not, we obtain
$$
A_1=\alpha_{1,1}\sum_{\substack{1\leq d_1d_2\leq i\\d_1\mid i,d_2\mid i+1\\ \gcd(d_1,d_2)=1}}\frac{\mu(d_1)\mu(d_2)}{(d_1d_2)^{k-1}}+O\big(i^{-1/2+\varepsilon}\tau(i)\tau(i+1)\big).
$$
Combining this with \eqref{eq:E(X_iX_i+1)=(1)}, we have
\begin{align}\label{eq:E(X_iX_i+1)=(2)}
\mathbb{E}({X}_i{X}_{i+1})=\sum_{\substack{1\leq d_1d_2\leq i\\d_1\mid i,d_2\mid i+1\\ \gcd(d_1,d_2)=1}}\frac{\mu(d_1)\mu(d_2)}{(d_1d_2)^{k-1}}+O\big(i^{-1/2+\varepsilon}\tau(i)\tau(i+1)\big).
\end{align}
It follows by Lemma \ref{lem:sum_{d_1d_2<n}=} and Lemma \ref{lem:the estimate about tau(n)} that
\begin{align}\label{eq:sum_{i leq n}E(X_iX_i+1)=}
\sum_{1\leq i\leq n}\mathbb{E}({X}_i{X}_{i+1})=n\prod_p\big(1-\frac{2}{p^k}\big)+O(n^{1/2+\varepsilon}).
\end{align}
Combining this with \eqref{eq:E(R_n)=} yields our required result.
\end{proof}

Now we are ready to deal with $\mathbb{V}(\overline{R}_{n,k})$.

\begin{proposition}\label{prop:V(R_n)=}
Let $q\geq 1$ and $k\geq 2$ be integers. Then for any $\varepsilon>0$, we have
$$
\mathbb{V}(\overline{R}_{n,k})=O_{\mathcal{A}, k,q,\varepsilon}(n^{-1/2+\varepsilon})
$$
as $n\rightarrow\infty$.
\end{proposition}

\begin{proof}
Note that
\begin{align}\label{eq:V(R(n))=E(R(n))^2-(E(R(n)))^2}
\mathbb{V}(\overline{R}_{n,k})=\mathbb{E}\big((\overline{R}_{n,k})^2\big)-\big(\mathbb{E}(\overline{R}_{n,k})\big)^2.
\end{align}
It follows by Proposition \ref{prop:E(R_n)=} that
\begin{align}\label{eq:(E(R(n)))^2=}
\big(\mathbb{E}(\overline{R}_{n,k})\big)^2&=\Big(\prod_p\big(1-\frac{2}{p^k}\big)+O(n^{-1/2+\varepsilon})\Big)^2\\
&=\Big(\prod_p\big(1-\frac{2}{p^k}\big)\Big)^2+O(n^{-1/2+\varepsilon})\nonumber.
\end{align}
We expand the square and obtain
\begin{align}\label{eq:E((R(n))^2)=sum_E(X_iX_i+1X_jX_j+1)+sum_E(X_i^2)}
\mathbb{E}\big((\overline{R}_{n,k})^2\big)=\frac{2}{n^2}\sum_{1\leq i<j\leq n}\mathbb{E}({X}_i{X}_{i+1}{X}_j{X}_{j+1})+\frac{1}{n^2}\sum_{1\leq i\leq n}\mathbb{E}\big(({{X}_i{X}_{i+1}})^2\big).
\end{align}
By \eqref{eq:sum_{i leq n}E(X_iX_i+1)=}, we have
\begin{align}\label{eq:sum E(X_iX_i+1)^2=}
\sum_{1\leq i\leq n}\mathbb{E}\big(({{X}_i{X}_{i+1}})^2\big)=\sum_{1\leq i\leq n}\mathbb{E}({X}_i{X}_{i+1})=O(n).
\end{align}
To deal with $\sum\limits_{1\leq i<j\leq n}\mathbb{E}({X}_i{X}_{i+1}{X}_j{X}_{j+1})$, by the definition of $X_i$, we have
$$
\mathbb{E}({X}_i{X}_{i+1}{X}_j{X}_{j+1})=\mathbb{P}({\bf p}_i,{\bf p}_{i+1},{\bf p}_{j},{\bf p}_{j+1}\ {\rm{are\ all\ visible}}).
$$
Using a similar argument as that yield \eqref{eq:E(X_iX_j)=} and \eqref{eq:E(X_iX_i+1)=(1)}, we have
\begin{align}\label{eq:E(X_iX_i+1X_jX_j+1)=AB}
\mathbb{E}({X}_i{X}_{i+1}{X}_j{X}_{j+1})=A(\mathcal{A},k,i) B(\mathcal{A},k,i,j),
\end{align}
where
$$
A(\mathcal{A},k,i)=A_1+\cdots+A_k
$$
and
$$
B(\mathcal{A},k,i,j)=B_1+\cdots+B_k
$$
with
$$
B_a=B_a(\mathcal{A},k,i,j):=\alpha_{1,a}\sum_{\substack{0\leq r_{t,1},\cdots,r_{t,k}\leq j_t\\r_{t,1}+\cdots+r_{t,k}=j_t\\ \forall 1\leq t\leq q\\ \gcd(s^{(1)}+r^{(1)},\cdots,s^{(k)}+r^{(k)})=1\\ \gcd(s^{(1)}+r^{(1)},\cdots,s^{(a)}+r^{(a)}+1,\cdots,s^{(k)}+r^{(k)})=1}}\prod_{t=1}^q P_{j_t,{\bf r}_t,{\boldsymbol{\alpha}}_t}
$$
for $1\leq a\leq k$, where $0\leq j_1,\cdots,j_t\leq j-i$ and $ j_1+\cdots+j_t=j-i$.
By a similar argument as that yields \eqref{eq:E(X_iX_i+1)=(2)}, we obtain
$$
A(\mathcal{A},k,i)=\sum_{\substack{1\leq d_1d_2\leq i\\d_1\mid i,d_2\mid i+1\\ \gcd(d_1,d_2)=1}}\frac{\mu(d_1)\mu(d_2)}{(d_1d_2)^{k-1}}+O\big(i^{-1/2+\varepsilon}\tau(i)\tau(i+1)\big)
$$
and
$$
B(\mathcal{A},k,i,j)=\sum_{\substack{1\leq d_1d_2\leq j\\d_1\mid j,d_2\mid j+1\\ \gcd(d_1,d_2)=1}}\frac{\mu(d_1)\mu(d_2)}{(d_1d_2)^{k-1}}+O\big((j-i)^{-1/2+\varepsilon}\tau(j)\tau(j+1)\big).
$$
Combining this with \eqref{eq:E(X_iX_i+1X_jX_j+1)=AB}, we expand the product and use the estimate
\begin{align}\label{eq:formula(1)}
\sum_{\substack{1\leq d_1d_2\leq i\\d_1\mid i,d_2\mid i+1\\ \gcd(d_1,d_2)=1}}\frac{\mu(d_1)\mu(d_2)}{(d_1d_2)^{k-1}}\ll\sum_{1\leq d_1d_2\leq i}\frac{1}{d_1d_2}\ll\sum_{1\leq m\leq i}\frac{\tau(m)}{m}\ll\log^2 i,
\end{align}
then

\begin{align}\label{eq:E}
\mathbb{E}({X}_i{X}_{i+1}{X}_j{X}_{j+1})&=\sum_{\substack{1\leq d_1d_2\leq i\\d_1\mid i,d_2\mid i+1\\ \gcd(d_1,d_2)=1}}\frac{\mu(d_1)\mu(d_2)}{(d_1d_2)^{k-1}}\sum_{\substack{1\leq d_1d_2\leq j\\d_1\mid j,d_2\mid j+1\\ \gcd(d_1,d_2)=1}}\frac{\mu(d_1)\mu(d_2)}{(d_1d_2)^{k-1}}\\
\nonumber&+O\big(i^{-1/2+\varepsilon}j^\varepsilon\tau(i)\tau(i+1)\big)+O\big(i^\varepsilon(j-i)^{-1/2+\varepsilon}\tau(j)\tau(j+1)\big).
\end{align}
Applying Lemma \ref{lem:the estimate about tau(n)} to the big-$O$ term yields
\begin{align}\label{eq:sum{i,j}E(X_iX_i+1X_jX_j+1)=(1)}
\sum_{1\leq i<j\leq n}\mathbb{E}({X}_i{X}_{i+1}{X}_j{X}_{j+1})=Y_{n,k}+O(n^{{3/2}+\varepsilon}),
\end{align}
where
$$
Y_{n,k}:=\sum_{1\leq i<j\leq n}\sum_{\substack{1\leq d_1d_2\leq i\\d_1\mid i,d_2\mid i+1\\ \gcd(d_1,d_2)=1}}\frac{\mu(d_1)\mu(d_2)}{(d_1d_2)^{k-1}}\sum_{\substack{1\leq d_1d_2\leq j\\d_1\mid j,d_2\mid j+1\\ \gcd(d_1,d_2)=1}}\frac{\mu(d_1)\mu(d_2)}{(d_1d_2)^{k-1}}.
$$

For $Y_{n,k}$, adding diagonal terms contributes an  error term bounded by $\ll n^{1+\varepsilon}$ by \eqref{eq:formula(1)}, we obtain
$$
Y_{n,k}=\frac{1}{2}\Big(\sum_{1\leq i\leq n}\sum_{\substack{1\leq d_1d_2\leq i\\d_1\mid i,d_2\mid i+1\\ \gcd(d_1,d_2)=1}}\frac{\mu(d_1)\mu(d_2)}{(d_1d_2)^{k-1}}\Big)^2+O(n^{1+\varepsilon}).
$$
Then by Lemma \ref{lem:sum_{d_1d_2<n}=}, we obtain
$$
Y_{n,k}=\frac{n^2}{2}\Big(\prod_p\big(1-\frac{2}{p^k}\big)\Big)^2+O(n^{1+\varepsilon}).
$$
Combining this with \eqref{eq:sum{i,j}E(X_iX_i+1X_jX_j+1)=(1)} gives
\begin{align}\label{eq:sum_{i,j}E(X_iX_i+1X_jX_j+1)=(2)}
\sum_{1\leq i<j\leq n}\mathbb{E}({X}_i{X}_{i+1}{X}_j{X}_{j+1})=\frac{n^2}{2}\Big(\prod_p\big(1-\frac{2}{p^k}\big)\Big)^2+O(n^{{3/2}+\varepsilon}).
\end{align}
Inserting \eqref{eq:sum E(X_iX_i+1)^2=} and\eqref{eq:sum_{i,j}E(X_iX_i+1X_jX_j+1)=(2)} into\eqref{eq:E((R(n))^2)=sum_E(X_iX_i+1X_jX_j+1)+sum_E(X_i^2)}, we obtain
\begin{align}
\mathbb{E}\big((\overline{R}_{n,k})^2\big)=\Big(\prod_p\big(1-\frac{2}{p^k}\big)\Big)^2+O(n^{{-1/2}+\varepsilon}).
\end{align}
This together with \eqref{eq:(E(R(n)))^2=} and \eqref{eq:V(R(n))=E(R(n))^2-(E(R(n)))^2} gives our desired result.

Now Theorem \ref{thm:R_n=} follows from Propositions \ref{prop:E(R_n)=} and \ref{prop:V(R_n)=}, and Lemma \ref{lem:lim S_n=u}.
\end{proof}

\section{Proof of Theorem \ref{thm:gamma_k{a;m}=}}
In this section, the implied constants in the big-$O$ or $\ll$ depend at most on $\mathcal{A},\ k,\ q$ and any given $\varepsilon>0$.

We first consider the expectation of $\overline{R}_{n,k}(a;m)$.
\begin{proposition}\label{prop:E(R(a,2^l))=}
Let $q,\ r\geq 1$ and $k\geq 2$ be integers. Then for any $\varepsilon>0$, we have
$$
\mathbb{E}\big(\overline{R}_{n,k}(a;2^r)\big)=\frac{1}{2^r}\prod_{p}\bigg(1-\frac{2}{p^k}\bigg)+O_{\mathcal{A},k,q,\varepsilon}(n^{-1/2+\varepsilon})
$$
as $n\rightarrow\infty$, where $p$ runs over all primes.
\end{proposition}

\begin{proof}
By \eqref{eq:E(X_iX_i+1)=(2)}, we obtain
\begin{equation}\label{eq:E(R(a,2^l))=}
\mathbb{E}\big(\overline{R}_{n,k}(a;2^r)\big)=\frac{1}{n}\sum_{\substack{1\leq i\leq n\\i\equiv a(\bmod 2^r)}}\mathbb{E}(X_iX_{i+1})=G_{n,k}(a;2^r)+O\big(n^{-1/2+\varepsilon}\big),
\end{equation}
where we used Lemma \ref{lem:the estimate about tau(n)} to estimate the big-$O$ term and
$$
G_{n,k}(a;2^r)=\frac{1}{n}\sum_{\substack{1\leq i\leq n\\i\equiv a(\bmod 2^r)}}\sum_{\substack{1\leq d_1d_2\leq i\\d_1\mid i,d_2\mid i+1\\ \gcd(d_1,d_2)=1}}\frac{\mu(d_1)\mu(d_2)}{(d_1d_2)^{k-1}}.
$$
For $G_{n,k}(a;2^r)$, we change the order of the summations, then
$$
G_{n,k}(a;2^r)=\frac{1}{n}\sum_{\substack{1\leq d_1d_2\leq n\\ \gcd(d_1,d_2)=1}}\frac{\mu(d_1)\mu(d_2)}{(d_1d_2)^{k-1}}\sum_{\substack{d_1d_2\leq i\leq n\\i\equiv a(\bmod 2^r)\\i\equiv0(\bmod d_1)\\i\equiv-1(\bmod d_2)}}1.
$$
Notice that the system of congruence equations
\begin{align}\label{eq:system of congruence equation(2)}
\begin{cases}
i\equiv a(\bmod\ 2^r)\\
i\equiv 0(\bmod\ d_1)\\
i\equiv -1(\bmod\ d_2)
\end{cases}
\end{align}
has solutions if and only if $\gcd(2^r,d_1)\mid a,\ \gcd(2^r,d_2)\mid a+1$ and $\gcd(d_1,d_2)=1$, and in which case, it has solutions
 $i\equiv c_2\big(\bmod\ \text{lcm}(2^r,d_1d_2)\big)$ for some $c_2\in\mathbb{Z}$.

If $a=0$ or $2^r-1$, we only consider the case $a=0$, and $a=2^r-1$ is the same. For $a=0$, note that \eqref{eq:system of congruence equation(2)} has solutions if and only if $\gcd(2^r,d_2)=\gcd(d_1,d_2)=1$, we have
$$
\begin{aligned}
 G_{n,k}(0;2^r)&=\frac{1}{n}\sum_{\substack{1\leq d_1d_2\leq n\\ \gcd(d_1,d_2)=1\\ \gcd(2^r,d_2)=1}}\frac{\mu(d_1)\mu(d_2)}{(d_1d_2)^{k-1}}\Big(\frac{\gcd(2^r,d_1d_2)n}{2^rd_1d_2}+O(1)\Big)\\
&=\frac{1}{2^r}\sum_{\substack{1\leq d_1d_2\leq n\\ \gcd(2^r,d_2)=1}}\frac{\gcd(2^r,d_1d_2)\mu(d_1d_2)}{(d_1d_2)^{k}}+O\Big(\frac{1}{n}\sum_{1\leq d_1d_2\leq n}\frac{1}{d_1d_2}\Big).
\end{aligned}
$$
Letting $d_1d_2=h$ and using Lemma \ref{lem:the estimate about tau(n)} to estimate the big-$O$ term, we obtain
$$
G_{n,k}(0;2^r)=\frac{1}{2^r}\sum_{1\leq h\leq n}\frac{\gcd(2^r,h)\mu(h)v(h,r)}{h^{k}}+O(n^{-1+\varepsilon}),
$$
where
$$
v(h,r)=\sum_{\substack{d\mid h\\ \gcd(2^r,d)=1}}1\ll\tau(h)\ll_{\varepsilon} h^{\varepsilon}.
$$
Thus by \eqref{eq:sum_{d>D}1/d^{1-theta}<<}, we obtain
$$
G_{n,k}(0;2^r)=\frac{1}{2^r}\sum_{h=1}^\infty\frac{\gcd(2^r,h)\mu(h)v(h,r)}{h^{k}}+O(n^{-1+\varepsilon}).
$$
By the fact that  $\gcd(2^r,h)$, $\mu(h)$ and  $v(h,r)$ are multiplicative in the variable $h$, we derive
\begin{equation}\label{eq:G(a,2^l) for a=0}
G_{n,k}(0;2^r)=\frac{1}{2^{r}}\prod_{p}\bigg(1-\frac{2}{p^k}\bigg)+O(n^{-1+\varepsilon}).
\end{equation}

If $a\neq0,\ 2^r-1$,  we only consider the case when $a$ is even, and it can be discussed in the same way for odd $a$. Suppose $\gcd(2^r,a)=2^b$, $1\leq b\leq r-1$, then \eqref{eq:system of congruence equation(2)} has solutions if and only if $\gcd(2^r,d_2)=\gcd(d_1,d_2)=1$ and $2^{b+1}\nmid d_1$. Therefore
$$
 G_{n,k}(a;2^r)=\frac{1}{n}\sum_{\substack{1\leq d_1d_2\leq n\\ \gcd(d_1,d_2)=1\\ \gcd(2^r,d_2)=1\\2^{b+1}\nmid d_1}}\frac{\mu(d_1)\mu(d_2)}{(d_1d_2)^{k-1}}\Big(\frac{\gcd(2^r,d_1d_2)n}{2^rd_1d_2}+O(1)\Big)=G_{n,k}(0;2^r),
$$
here we used the fact that $\mu(d_1)=0$ for $2^{b+1}\mid d_1$.
Combining this with \eqref{eq:G(a,2^l) for a=0} and  \eqref{eq:E(R(a,2^l))=} gives our results.
\end{proof}

\begin{proposition}\label{prop:E(R(a,p_1))=}
Let $q\geq 1$ and $k\geq 2$ be integers. Then for any prime $p_1\geq 3$ and any $\varepsilon>0$, we have
$$
\mathbb{E}\big(\overline{R}_{n,k}(0\ {\rm{or}}\ p_1-1;p_1)\big)=\frac{p_1^{k-1}-1}{p_1^{k}-2}\prod_{p}\bigg(1-\frac{2}{p^k}\bigg)+O_{\mathcal{A},k,q,\varepsilon}(n^{-1/2+\varepsilon})
$$
and
$$
\mathbb{E}\big(\overline{R}_{n,k}(a;p_1)\big)=\frac{p_1^{k-1}}{p_1^{k}-2}\prod_{p}\bigg(1-\frac{2}{p^k}\bigg)+O_{\mathcal{A},k,q,\varepsilon}(n^{-1/2+\varepsilon})\quad \text{if}~a\neq 0,\ p_1-1
$$
as $n\rightarrow\infty$, where $p$ runs over all primes.
\end{proposition}

\begin{proof}
We recall \eqref{eq:E(X_iX_i+1)=(2)} that
\begin{equation}\label{eq:E(R(a,p_1))=}
\mathbb{E}\big(\overline{R}_{n,k}(a;p_1)\big)=G_{n,k}(a;p_1)+O\big(n^{-1/2+\varepsilon}\big),
\end{equation}
where
$$
G_{n,k}(a;p_1)=\frac{1}{n}\sum_{\substack{1\leq i\leq n\\i\equiv a(\bmod p_1)}}\sum_{\substack{1\leq d_1d_2\leq i\\d_1\mid i,d_2\mid i+1\\ \gcd(d_1,d_2)=1}}\frac{\mu(d_1)\mu(d_2)}{(d_1d_2)^{k-1}}=\frac{1}{n}\sum_{\substack{1\leq d_1d_2\leq n\\ \gcd(d_1,d_2)=1}}\frac{\mu(d_1)\mu(d_2)}{(d_1d_2)^{k-1}}\sum_{\substack{d_1d_2\leq i\leq n\\i\equiv a(\bmod p_1)\\i\equiv0(\bmod d_1)\\i\equiv-1(\bmod d_2)}}1.
$$

If $a=0$ or $p_1-1$, we discuss in the same way as in the proof of Proposition \ref{prop:E(R(a,2^l))=}, hence
\begin{equation}\label{eq:G(a,p_1) for a=0 or p_1-1}
G_{n,k}(a;p_1)=\frac{p_1^{k-1}-1}{p_1^{k}-2}\prod_{p}\bigg(1-\frac{2}{p^k}\bigg)+O(n^{-1+\varepsilon}).
\end{equation}

If $a\neq0 ~\text{or}\ p_1-1$, we know that
$$
\begin{cases}
i\equiv a(\bmod\ p_1)\\
i\equiv 0(\bmod\ d_1)\\
i\equiv -1(\bmod\ d_2)
\end{cases}
$$
has solutions if and only if $\gcd(p_1,d_1)=\gcd(p_1,d_2)=\gcd(d_1,d_2)=1$, then by the Chinese Reminder Theorem, we obtain
$$
G_{n,k}(a;p_1)=\frac{1}{n}\sum_{\substack{1\leq d_1d_2\leq n\\ \gcd(d_1,d_2)=1\\ \gcd(p_1,d_1)=1\\ \gcd(p_1,d_2)=1}}\frac{\mu(d_1)\mu(d_2)}{(d_1d_2)^{k-1}}\Big(\frac{n}{p_1d_1d_2}+O(1)\Big).
$$
Letting $d_1d_2=h$, we derive
$$
G_{n,k}(a;p_1)=\frac{1}{p_1}\sum_{\substack{1\leq h\leq n\\ \gcd(p_1,h)=1}}\frac{\mu(h)\tau(h)}{h^k}+O(n^{-1+\varepsilon}).
$$
If we consider the multiplicative function
$$
g(h):=\left\{
\begin{aligned}
&\frac{\mu(h)\tau(h)}{h^{k}},\ \ {\rm{if}}\ \gcd(p_1,h)=1,\\
&0,\ \ \ \ \ \ \ \ \ \ \ \ \ \ \ \rm{otherwise},
\end{aligned}
\right.
$$
then by Lemma \ref{lem:the estimate about tau(n)}, we have
$$
G_{n,k}(a;p_1)=\frac{1}{p_1}\sum_{1\leq h\leq n}g(h)+O(n^{-1+\varepsilon})=\frac{1}{p_1}\sum_{h=1}^\infty g(h)+O(n^{-1+\varepsilon}),
$$
which gives
\begin{equation}\label{eq:G(a,p_1) for a otherwise}
G_{n,k}(a;p_1)=\frac{p_1^{k-1}}{p_1^{k}-2}\prod_{p}\bigg(1-\frac{2}{p^k}\bigg)+O(n^{-1+\varepsilon}).
\end{equation}
Plugging \eqref{eq:G(a,p_1) for a=0 or p_1-1} and \eqref{eq:G(a,p_1) for a otherwise} into \eqref{eq:E(R(a,p_1))=} yields our desired results.
\end{proof}

Now we deal with the variance of $\overline{R}_{n,k}(a;m)$.
\begin{proposition}\label{prop:V(R(a,m))=}
Let $q,\ r\geq 1$ and $k\geq 2$ be integers. Then for any $\varepsilon>0$, we have
$$
\mathbb{V}\big(\overline{R}_{n,k}(a;m)\big)=O_{\mathcal{A},k,q,\varepsilon}(n^{-1/2+\varepsilon})
$$
as $n\rightarrow\infty$.
\end{proposition}

\begin{proof}
 Similar as in the proof of Proposition \ref{prop:V(R_n)=}, we only need to handle $\big(\mathbb{E}\big(\overline{R}_{n,k}(a;m)\big)\big)^2$ and $\mathbb{E}\big(\big(\overline{R}_{n,k}(a;m)\big)^2\big)$.

 By Proposition \ref{prop:E(R(a,2^l))=} and Proposition \ref{prop:E(R(a,p_1))=}, we have
 $$
 \big(\mathbb{E}\big(\overline{R}_{n,k}(a;m)\big)\big)^2=\big(G_{n,k}(a,m)\big)^2+O(n^{-1/2+\varepsilon}).
 $$
 For $\mathbb{E}\big(\big(\overline{R}_{n,k}(a;m)\big)^2\big)$, we write
$$
\mathbb{E}\big(\big(\overline{R}_{n,k}(a;m)\big)^2\big)=\frac{2}{n^2}\sum_{\substack{1\leq i<j\leq n\\i,j\equiv a(\bmod m)}}\mathbb{E}({X}_iX_{i+1}{X}_jX_{j+1})+\frac{1}{n^2}\sum_{\substack{1\leq i\leq n\\i\equiv a(\bmod m)}}\mathbb{E}\big(({X}_iX_{i+1})^2\big).
$$
By Propositions \ref{prop:E(R(a,2^l))=} and \ref{prop:E(R(a,p_1))=}, the second term on the right hand side of the above formula is $\ll 1/n$, and for the first term, we recall \eqref{eq:E} and add the diagonal terms, then
$$
\begin{aligned}
\frac{2}{n^2}\sum_{\substack{1\leq i<j\leq n\\i,j\equiv a(\bmod m)}}\mathbb{E}({X}_iX_{i+1}{X}_jX_{j+1})&=\Big(\frac{1}{n}\sum_{\substack{1\leq i\leq n\\i\equiv a(\bmod m)}}\sum_{\substack{1\leq d_1d_2\leq i\\d_1\mid i,d_2\mid i+1\\ \gcd(d_1,d_2)=1}}\frac{\mu(d_1)\mu(d_2)}{(d_1d_2)^{k-1}}\Big)^2+O(n^{-1/2+\varepsilon})\\
&=\big(G_{n,k}(a,m)\big)^2+O(n^{-1/2+\varepsilon}),
\end{aligned}
$$
which gives
$$
\mathbb{E}\big(\big(\overline{R}_{n,k}(a;m)\big)^2\big)=\big(G_{n,k}(a,m)\big)^2+O(n^{-1/2+\varepsilon}).
$$

Now Theorem \ref{thm:gamma_k{a;m}=} follows from Propositions \ref{prop:E(R(a,2^l))=}, \ref{prop:E(R(a,p_1))=} and \ref{prop:V(R(a,m))=}, and Lemma \ref{lem:lim S_n=u}.
\bigskip
\end{proof}

\end{document}